\documentclass[12pt, reqno]{amsart}
\usepackage{latexsym}
\usepackage{mathtools}
\usepackage{amsmath}
\usepackage{amssymb}
\usepackage{mathrsfs}
\usepackage[abbrev]{amsrefs}
\usepackage{graphicx}
\usepackage[all]{xy}
\usepackage{enumerate}

\allowdisplaybreaks

\newtheorem{Theorem}{Theorem}[section]
\newtheorem{Lemma}[Theorem]{Lemma}
\newtheorem{Proposition}[Theorem]{Proposition}
\newtheorem{Corollary}[Theorem]{Corollary}
\theoremstyle{remark}
\newtheorem{Remark}[Theorem]{Remark}

\newtheorem{Question}[Theorem]{Question}



\begin{document}

\title[Embeddability of $\ell_p$ and bases in Lipschitz free $p$-spaces]{Embeddability of $\ell_p$ and bases \\ in Lipschitz free $p$-spaces for $0<p\le 1$}

\author[F. Albiac]{Fernando Albiac}
\address{Mathematics Department--InaMat \\
Universidad P\'ublica de Navarra\\
Campus de Arrosad\'{i}a\\
Pamplona\\
31006 Spain}
\email{fernando.albiac@unavarra.es}

\author[J. L. Ansorena]{Jos\'e L. Ansorena}
\address{Department of Mathematics and Computer Sciences\\
Universidad de La Rioja\\
Logro\~no\\
26004 Spain}
\email{joseluis.ansorena@unirioja.es}

\author[M. C\'uth]{Marek C\'uth}
\address{Faculty of Mathematics and Physics, Department of Mathematical Analysis\\
Charles University\\
186 75 Praha 8\\
Czech Republic}
\email{cuth@karlin.mff.cuni.cz}

\author[M. Doucha]{Michal Doucha}
\address{Institute of Mathematics\\
Czech Academy of Sciences\\
\v Zitn\'a 25\\
115 67 Praha 1\\
Czech Republic}
\email{doucha@math.cas.cz}

\begin{abstract}
Our goal in this paper is to continue the study initiated by the authors
in \cite{AACD2018} of the geometry of the Lipschitz free $p$-spaces
over quasimetric spaces for $0<p\le 1$, denoted $\mathcal{F}_{p}(
\mathcal{M})$. Here we develop new techniques to show that, by analogy
with the case $p=1$, the space $\ell _{p}$ embeds isomorphically in
$\mathcal{F}_{p}(\mathcal{M})$ for $0<p<1$. Going further we see that
despite the fact that, unlike the case $p=1$, this embedding need not
be complemented in general, complementability of $\ell _{p}$ in a
Lipschitz free $p$-space can still be attained by imposing certain
natural restrictions to $\mathcal{M}$. As a by-product of our discussion
on bases in $\mathcal{F}_{p}([0,1])$,
we obtain examples of $p$-Banach spaces for $p < 1$ that are not based
on a trivial modification of Banach spaces, which possess a basis but fail to have an unconditional basis.
\end{abstract}

\thanks{F. Albiac acknowledges the support of the Spanish Ministry for Economy and Competitivity under Grant MTM2016-76808-P for
\emph{Operators, lattices, and structure of Banach spaces}, and the support of the Spanish Ministry for Science, Innovation, and Universities under Grant  PGC2018-095366-B-I00 for \emph{An\'{a}lisis Vectorial, Multilineal y Aproximaci\'{o}n}. He would also like to thank the Isaac Newton Institute for Mathematical Sciences, Cambridge (United Kingdom), for
support and hospitality during the program \emph{Approximation, Sampling and Compression in Data Science}, where work on this paper was undertaken. This work was supported by EPSRC grant number EP/R014604/1. J.~L. Ansorena acknowledges the support of the Spanish Ministry for Science, Innovation, and Universities under Grant PGC2018-095366-B-I00 for \emph{An\'{a}lisis Vectorial, Multilineal y Aproximaci\'{o}n}. M. C\'{u}th has been supported by Charles University Research program No. UNCE/SCI/023 and by the Research grant GA\v{C}R 17-04197Y. M. Doucha was supported by the GA\v{C}R project 19-05271Y and RVO: 67985840.}

\maketitle

\section{Introduction}\label{sec1}
\noindent

Given a pointed $p$-metric space $\mathcal{M}$ ($0<p\le 1$) it is
possible to construct a unique $p$-Banach space $\mathcal{F}_{p}(
\mathcal{M})$ in such a way that $\mathcal{M}$ embeds isometrically in
$\mathcal{F}_{p}(\mathcal{M})$, and for every $p$-Banach space $X$ and
every Lipschitz map $f\colon \mathcal{M}\to X$ that maps the base point
$0$ in $\mathcal{M}$ to $0\in X$ there exists a unique linear map
$T_{f}\colon \mathcal{F}_{p}(\mathcal{M})\to X$ with $\Vert T_{f}
\Vert =\mathrm{Lip}(f)$. The space $\mathcal{F}_{p}(\mathcal{M})$ is
known as the \emph{Lipschitz free $p$-space} (or the \emph{Arens-Eells
$p$-space\emph{)} over $\mathcal{M}$}. Lipschitz free $p$-spaces provide a
canonical linearization process of Lipschitz maps between $p$-metric
spaces: if we identify (through the isometric embedding $
\delta _{\mathcal{M}}\colon \mathcal{M}\to \mathcal{F}_{p}(\mathcal{M})$)
a $p$-metric space $\mathcal{M}$ with a subset of $\mathcal{F}_{p}(
\mathcal{M})$, then any Lipschitz map $f$ from a $p$-metric space
$\mathcal{M}_{1}$ to a $p$-metric space $\mathcal{M}_{2}$ which maps the
base point in $\mathcal{M}_{1}$ to the base point in $\mathcal{M}_{2}$
extends to a continuous linear map $L_{f}\colon \mathcal{F}_{p}(
\mathcal{M}_{1})\to \mathcal{F}_{p}(\mathcal{M}_{2})$ and $\Vert L
_{f}\Vert =\mathrm{Lip}(f)$.

Lipschitz free $p$-spaces were introduced in
\cite{AlbiacKalton2009}, where they were used to provide for every for
$0<p<1$ a couple of \emph{separable} $p$-Banach spaces which are
Lipschitz-isomorphic without being linearly isomorphic. The study of the
structure of the spaces $\mathcal{F}_{p}(\mathcal{M})$, however, has not
been tackled until recently in \cite{AACD2018}, where we refer the
reader for terminology and background. These spaces constitute a nice
family of new $p$-Banach spaces which are easy to define but whose
geometry seems to be difficult to understand. To carry out this task
successfully one hopes to be able to count on ``natural'' structural
results involving free $p$-spaces over subsets of $\mathcal{M}$. In
\cite{AACD2018}*{\S 6} we analyzed this premise and confirmed an
unfortunate recurrent pattern in quasi-Banach spaces: the lack of tools
can be an important stumbling block in the development of the theory.
However, as we see here, not everything is lost and we still can develop
specific methods that permit to shed light onto the geometry of
$\mathcal{F}_{p}(\mathcal{M})$.

This paper is a continuation of the study initiated by the authors in
\cite{AACD2018}. Our aim is to delve deeper into the structure of this
class of spaces and address very natural questions that arise by analogy
with the case $p=1$. Needless to say, the extension of such results is
far from straightforward since the techniques used for metric and Banach
spaces break down when the local convexity is lifted. As a consequence,
our work provides a new view (and often also rather different proofs)
of the structural results already known for the standard Lipschitz free
spaces.

To that end, we start in Section~\ref{Sect2.1} with a method, which is an
extension of known results for $p=1$ to the more general case of
$p\in (0,1]$, for constructing $p$-metric spaces by taking certain sums.
The other two subsections of Section~\ref{Section:Comp} contain results that
are new even for the case when $p=1$. In Section~\ref{Sect2.2} we address
metric quotients. The main application here is probably that some
Sobolev spaces are isometric to certain Lipschitz free spaces (see
Theorem~\ref{thm:sobolev}). In Section~\ref{Sect2.3} we describe the kernel
of a projection induced by a Lipchitz retraction. This was already
considered in \cite{HN17}*{Proposition 1} for $p=1$, but we obtain a
different description using metric quotients (see
Theorem~\ref{thm:retr}).

It is known that $\ell _{1}$ is isomorphic to a complemented subspace of
$\mathcal{F}(\mathcal{M})$ whenever $\mathcal{M}$ is an infinite metric
space (see \cite{CDW2016} and \cite{HN17}). This important
structural property does not carry over, in general, to free $p$-spaces
over quasimetric spaces when $p<1$. Indeed, $\ell _{p}$ (whose dual is
$\ell _{\infty }$) fails to be complemented in $L_{p}[0,1]$ (which is a
Lipschitz free $p$-space by \cite{AACD2018}*{Theorem 4.13}) since
$L_{p}[0,1]^{\ast }=\{0\}$. However, as we will see in Section~\ref{Section:Comp}, 
there are conditions on $\mathcal{M}$ which
ensure that $\ell _{p}$ does embed complementably into $\mathcal{F}
_{p}(\mathcal{M})$ for every $0<p\le 1$. The most important case occurs
when $\mathcal{M}$ is  an infinite
metric space.  We also extend the results from \cite{HN17}
concerning Lipschitz free $p$-spaces over different nets of a $p$-metric
space (see Proposition~\ref{prop:nets}).

The question we tackle in Section~\ref{Sect3} is whether, by analogy
with the case $p=1$, we can guarantee that $\ell _{p}$ embeds
isomorphically in any $\mathcal{F}_{p}(\mathcal{M})$ for $0<p<1$. The
answer is positive, but in order to prove it we must develop a
completely new set of techniques, specific of the nonlocally convex case
(which, by the way, also work for $p=1$). Some of our results in this section such as 
Theorem~\ref{thm:main2} are new  even for the case
$p=1$.

In the last, and partially independent Section~\ref{Sect:bases}, we
investigate Schauder bases in $\mathcal{F}_{p}(\mathbb{N})$ and
$\mathcal{F}_{p}([0,1])$ when $0<p<1$. One of the main results here is that
$\mathcal{F}_{p}([0,1])$ has a Schauder basis.
These provide examples of $p$-Banach spaces for $p<1$ with a basis which do not have an unconditional basis and are not a trivial modification/deformation of a Banach space such as $L_{1}\oplus \ell _{p}$, thus reinforcing the theoretical usefulness of Lipschitz free $p$-spaces for $p < 1$.

Throughout this article we use standard notation in Banach space theory
as can be found in \cite{AlbiacKalton2016}. We refer the reader to
\cites{GodefroyKalton2003,Weaver2018} for basic facts on Lipschitz free
spaces and some of their uses, and to \cite{KPR1984} for background
on quasi-Banach spaces.

Let us start with an analogue to \cite{HN17}*{Proposition 2} for
$p$-metric spaces.

\begin{Lemma}
\label{lem:isoWithELLPSum}Let $(\mathcal{M},d,0)$ be a pointed $p$-metric space, $p\in (0,1]$.
Suppose that $(\mathcal{M}_{\alpha })_{\alpha \in \Delta }$ is a family
of subsets of $\mathcal{M}$ such that $0\in \mathcal{M}_{\alpha }$ for
every $\alpha \in \Delta $ and 
$(\mathcal{M}_{\alpha }\setminus \{0\})_{\alpha \in \Delta }$ is a partition $\mathcal{M}\setminus \{0\}$.
Suppose further that there exists $K\ge 1$ such that for all
$x\in \mathcal{M}_{\alpha }$ and all $y\in \mathcal{M}_{\beta }$ with
$\alpha ,\beta \in \Delta $, $\alpha \neq \beta $, we have
\begin{equation}
\label{eq:almostEllPSum}
K^{p} d^{p}(x,y)\geq d^{p}(x,0) + d^{p}(y,0).
\end{equation}
Then the map
\begin{equation*}
T\colon \left (\bigoplus _{\alpha \in \Delta } \mathcal{F}_{p}(
\mathcal{M}_{\alpha })\right )_{p} \to \mathcal{F}_{p}(\mathcal{M}),
\quad (x_{\alpha })_{\alpha \in \Delta } \mapsto \sum _{\alpha \in
\Delta } L_{\alpha }(x_{\alpha }),
\end{equation*}
where $L_{\alpha }$ denotes the canonical map from $\mathcal{F}_{p}(
\mathcal{M}_{\alpha })$ into $\mathcal{F}_{p}(\mathcal{M})$, is an onto
isomorphism. Quantitatively, $\Vert T\Vert \le 1$ and $\Vert T^{-1}
\Vert \le K$.
\end{Lemma}

In order to prove Lemma~\ref{lem:isoWithELLPSum} we will use the notion
of $p$-norming set. We say that a subset $Z$ of a $p$-Banach
space $X$ is \emph{$p$-norming with constants $c_{1}$ and $c_{2}$} if
$c_{1}^{-1} Z$ is contained in the unit closed ball $B_{X}$ of $X$ and
$c_{2}^{-1} B_{X}$ is contained in the smallest absolutely $p$-convex
closed set containing $Z$, denoted by
$\operatorname{\mathrm{co}}_{p}(Z)$ (see \cite{AACD2018}*{\S 2.3}).

\begin{proof}[Proof of Lemma~\ref{lem:isoWithELLPSum}]
Let $\delta _{\alpha }$ be the canonical embedding of $\mathcal{M}_{
\alpha }$ into $(\bigoplus _{\alpha \in \Delta } \mathcal{F}_{p}(
\mathcal{M}_{\alpha }))_{p}$. By \cite{AACD2018}*{Corollary 4.11}, the
set
\begin{equation*}
\bigcup _{\alpha \in \Delta } \left \{  \frac{\delta _{\alpha }(y)-
\delta _{\alpha }(x)}{d(x,y)} \colon x,y \in \mathcal{M}_{\alpha },\, x
\neq y\right \}
\end{equation*}
is an isometric $p$-norming set for $(\bigoplus _{\alpha \in \Delta }
\mathcal{F}_{p}(\mathcal{M}_{\alpha }))_{p}$. Hence, by
\cite{AACD2018}*{Lemmas 2.6 and 2.7}, we must prove that
\begin{equation*}
Z:=\bigcup _{\alpha \in \Delta } \left \{  \frac{\delta _{\mathcal{M}}(y)-
\delta _{\mathcal{M}}(x)}{d(x,y)} \colon x,y \in \mathcal{M}_{\alpha },
x\neq y\right \}
\end{equation*}
is a $p$-norming set for $\mathcal{F}_{p}(\mathcal{M})$ with constants
$1$ and $K$. To that end, invoking again \cite{AACD2018}*{Corollary~4.11}, it suffices to prove that
\begin{equation*}
Z_{1}:=\left \{
\frac{\delta _{\mathcal{M}}(y)-\delta _{\mathcal{M}}(x)}{d(x,y)} \colon
x\in \mathcal{M}_{\alpha }, \, y\in \mathcal{M}_{\beta }, \, \alpha
\neq \beta \right \}  \subseteq K\operatorname{\mathrm{co}}_{p}(Z).
\end{equation*}
Let $x,y\in \mathcal{M}$ and pick $\alpha $ and $\beta $ such that
$x\in \mathcal{M}_{\alpha }$ and $y\in \mathcal{M}_{\beta }$. We have
\begin{equation*}
\frac{\delta _{\mathcal{M}}(y)-\delta _{\mathcal{M}}(x)}{d(x,y)}=\lambda
\frac{\delta _{\mathcal{M}}(y)}{d(y,0)}+ \mu
\frac{\delta _{\mathcal{M}}(x)}{d(x,0)},
\end{equation*}
where
\begin{equation*}
\lambda =\frac{d(y,0)}{d(x,y)} \text{ and } \mu =-
\frac{d(x,0)}{d(x,y)}.
\end{equation*}
Since $|\lambda |^{p}+|\mu |^{p}\le K^{p}$, we are done.
\end{proof}

\subsection{Sums of quasimetric spaces}\label{Sect2.1}
We now introduce a method inspired by Lemma~\ref{lem:isoWithELLPSum} for
building quasimetric spaces. Given a family $(M_{\alpha },d_{\alpha },0)_{\alpha
\in \Delta }$ of pointed $p$-metric spaces we consider
\begin{equation*}
\left (\bigoplus _{\alpha \in \Delta } \mathcal{M}_{\alpha }\right )
_{p}
=\Big \{(x_{\alpha })_{\alpha \in \Delta }\in
\Pi _{\alpha \in \Delta } \mathcal{M}_{\alpha }\colon (d_{\alpha }(x
_{\alpha },0))_{\alpha \in \Delta }\in \ell _{p}(\Delta )\Big \}.
\end{equation*}
Since $(d_{\alpha }(x^{1}_{\alpha },x^{2}_{\alpha }))_{\alpha \in
\Delta }\in \ell _{p}(\Delta )$ whenever $(x^{i}_{\alpha })_{\alpha
\in \Delta }\in (\bigoplus _{\alpha \in \Delta }
\mathcal{M}_{\alpha })_{p}$, $i=1$, $2$, we can safely define a
$p$-distance $d$ on $(\bigoplus _{n=1}^{\infty }\mathcal{M}_{n})_{p}$ by
\begin{equation*}
d((x^{1}_{\alpha })_{\alpha \in \Delta }, (x^{2}_{\alpha })_{\alpha
\in \Delta })
=\left (\sum _{\alpha \in \Delta } d_{\alpha }^{\, p}(x
^{1}_{\alpha },x^{2}_{\alpha })\right )^{1/p}.
\end{equation*}
The base point of $\big ((\bigoplus _{\alpha \in \Delta } \mathcal{M}
_{\alpha })_{p},d\big )$ will be the element whose entries are the base
points of each $p$-metric space $\mathcal{M}_{\alpha }$. We consider the
subset
\begin{equation*}
\maltese _{\alpha \in \Delta } \mathcal{M}_{\alpha }=\left \{  (x_{
\alpha })_{\alpha \in \Delta } \in \Pi _{\alpha \in \Delta }
\mathcal{M}_{\alpha }\colon |\{ \alpha \colon x_{\alpha }\neq0\}|
\le 1\right \}
\end{equation*}
of the pointed $p$-metric space $\big ((\bigoplus _{\alpha \in \Delta }
\mathcal{M}_{\alpha })_{p},d,0\big )$. If $\mathcal{M}_{\alpha }=
\mathcal{M}$ for every $\alpha \in \Delta $ we put $\left (\bigoplus _{\alpha
\in \Delta } \mathcal{M}_{\alpha }\right )_{p}=\ell _{p}(\mathcal{M},
\Delta )$ and $\maltese _{\alpha \in \Delta } M_{\alpha }=\maltese (
\mathcal{M},\Delta )$ (respectively, $\ell _{p}(\mathcal{M})$ and
$\maltese (\mathcal{M})$ in the case when $\Delta =\mathbb{N}$). If
$\Delta $ is finite (for instance $\Delta =\{a,b\}$) we write
$\left (\bigoplus _{\alpha \in \Delta } \mathcal{M}_{\alpha }\right )
_{p}=\mathcal{M}_{a}\oplus \mathcal{M}_{b}$ and $
\maltese _{\alpha \in \Delta } M_{\alpha }=\mathcal{M}_{a}\maltese
\mathcal{M}_{b}$.

The spaces $\maltese _{\alpha \in \Delta } M_{\alpha }$ were considered
in \cite{DuFe} in the metric space setting, although its authors
preferred to use the (equivalent) supremum norm instead of the
$\ell _{1}$-norm to combine the spaces. Notice that the canonical
embedding $I_{\alpha }$ of $\mathcal{M}_{\alpha }$ into $
\maltese _{\alpha \in \Delta } \mathcal{M}_{\alpha }$ is an isometry,
that $(I_{\alpha }(\mathcal{M}_{\alpha }\setminus \{0\}))_{\alpha
\in \Delta }$ is a partition of $\maltese _{\alpha \in \Delta } M_{
\alpha }\setminus \{0\}$ and that, if $\alpha \neq\beta $ and $x$,
$y\in \maltese _{\alpha \in \Delta } \mathcal{M}_{\alpha }$,
\begin{equation*}
d(I_{\alpha }(x),I_{\beta }(y))
=\left (d^{p}(I_{\alpha }(x),0) + d
^{p}(I_{\beta }(y),0)\right )^{1/p}.
\end{equation*}
Hence, Lemma~\ref{lem:isoWithELLPSum} immediately gives the following.
\begin{Lemma}
\label{lem:2}
Let $(\mathcal{M}_{\alpha },d_{\alpha },0)_{\alpha \in \Delta }$ be a
family of pointed $p$-metric spaces, $p\in (0,1]$. Then
\begin{equation*}
\mathcal{F}_{p}(\maltese _{\alpha \in \Delta } \mathcal{M}_{\alpha })
\simeq \left (\bigoplus _{\alpha \Delta } \mathcal{F}_{p}(\mathcal{M}
_{\alpha })\right )_{p}.
\end{equation*}
To be precise, the canonical map $L_{\maltese }$ given by
\begin{equation*}
(x_{\alpha })_{\alpha \in \Delta } \mapsto \sum _{\alpha \in \Delta } L
_{I_{\alpha }}(x_{\alpha })
\end{equation*}
is an isometry.
\end{Lemma}

Let us next present a few applications of Lipschitz free $p$-spaces over
Banach spaces of continuous functions.
\begin{Proposition}
\label{AJlA}
For every $0<p\le 1$ we have $\mathcal{F}_{p}(c_{0})\simeq \ell _{p}(
\mathcal{F}_{p}(c_{0}))$.
\end{Proposition}

\begin{proof}
We just need to mimic the proof of \cite{DuFe}*{Proposition 4} with
the aid of Lemma~\ref{lem:2}.
\end{proof}

In \cite{DuFe}, Dutrieux and Ferenczi provided an example of two
non-Lipschitz isomorphic Banach spaces whose corresponding Lipschitz
free spaces are isomorphic. The following result is a strengthening of
the main result from \cite{DuFe}, in the sense that we make it
extensive to Lipschitz free $p$-spaces for $0<p<1$. As in
\cite{DuFe}, our approach relies on the fact that if $K$ is an
uncountable compact metric space then $\mathcal{C}(K)$ and $c_{0}$ are
not Lipschitz isomorphic (see \cite{JLS1996}).

\begin{Theorem}
Let $K$ be any infinite compact metric space. Then for every
$0<p\leq 1$ we have
\begin{equation*}
\mathcal{F}_{p}(\mathcal{C}(K))\simeq \mathcal{F}_{p}(c_{0})
\end{equation*}
(with uniformly bounded Banach-Mazur distance).
\end{Theorem}

\begin{proof}
On the one hand $c_{0}$ is complemented in $\mathcal{C}(K)$ (see
e.g. \cite{AlbiacKalton2016}*{Corollary 2.5.9}) so, in particular,
$c_{0}$ is a Lipschitz retract of $\mathcal{C}(K)$. On the other hand,
since $\mathcal{C}(K)$ is Lipschitz isomorphic to a subspace of
$c_{0}$ by \cite{Aharoni1974}, and $\mathcal{C}(K)$ is an absolute
Lipschitz retract (see \cite{Lindenstrauss1964}*{Theorem 6}) it
follows that $\mathcal{C}(K)$ is a Lipschitz retract of $c_{0}$. Then,
by \cite{AACD2018}*{Lemma~4.19}, $\mathcal{F}_{p}(c_{0})$ is
complemented in $\mathcal{F}_{p}(\mathcal{C}(K))$ and $\mathcal{F}
_{p}(\mathcal{C}(K))$ is complemented in $\mathcal{F}_{p}(c_{0})$. Now,
taking into account Proposition~\ref{AJlA}, Pe\l czy\'{n}ski's
decomposition method yields $\mathcal{F}_{p}(\mathcal{C}(K))\simeq
\mathcal{F}_{p}(c_{0})$. Finally, we note that since all constants
involved are independent of the compact space $K$, so is the isomorphism
constant.
\end{proof}

If $\mathcal{N}$ is a subset of a metric space $\mathcal{M}$ and
$\jmath $ denotes the inclusion from $\mathcal{N}$ into $\mathcal{M}$,
the canonical linear map $L_{\jmath }\colon \mathcal{F}(\mathcal{N})
\to \mathcal{F}(\mathcal{M})$ is an isometric embedding. Although this
result does not carry over, in general, to Lipschitz free $p$-spaces for
$p<1$ (see \cite{AACD2018}*{Theorem~6.1}), it is an open question
whether $L_{\jmath }\colon \mathcal{F}_{p}(\mathcal{N})\to
\mathcal{F}_{p}(\mathcal{M})$ is always an \emph{isomorphism} for
$p<1$ (see \cite{AACD2018}*{Question 6.2}).
Lemma~\ref{L24} sheds some light onto this matter.

\begin{Lemma}
\label{L24}
Let $0<p\le q\le 1$. We have the following dichotomy: Either there is
a $q$-metric space $\mathcal{M}$ and a subset $\mathcal{N}$ of
$\mathcal{M}$ for which $L_{\jmath }\colon \mathcal{F}_{p}(
\mathcal{N})\to \mathcal{F}_{p}(\mathcal{M})$ is not an isomorphism, or
there is a universal constant $C$ (depending only on $p$ and $q$) such
that $\Vert L_{\jmath }^{-1}\Vert \le C$ for every $q$-metric
space space $\mathcal{M}$ and every subset $\mathcal{N}$ of $\mathcal{M}$.
\end{Lemma}
\begin{proof}
Assume, by contradiction, that the lemma fails to be true. Then
$L_{\jmath }\colon \mathcal{F}_{p}(\mathcal{N})\to \mathcal{F}_{p}(
\mathcal{M})$ is always an isomorphism, while $\Vert L_{\jmath }^{-1}
\Vert $ can be arbitrarily large. Thus there is a sequence $(
\mathcal{M}_{n})_{n=1}^{\infty }$ of $q$-metric spaces and a sequence
$(\mathcal{N}_{n})_{n=1}^{\infty }$, where each $\mathcal{N}_{n}$ is a
subset of $\mathcal{M}_{n}$, such that, if $L_{n}$ denotes the canonical
linear map from $\mathcal{F}_{p}(\mathcal{N}_{n})$ into $\mathcal{F}
_{p}(\mathcal{M}_{n})$ then $\sup _{n}\Vert L_{n}^{-1} \Vert =\infty $.

The $q$-metric space $\mathcal{N}:=\maltese _{n=1}^{\infty }
\mathcal{N}_{n}$ is a subset of the $q$-metric space $\mathcal{M}:=
\maltese _{n=1}^{\infty }\mathcal{M}_{n}$. By assumption, the canonical
map from $\mathcal{F}_{p}(\mathcal{N})$ into $\mathcal{F}_{p}(
\mathcal{M})$ is an isomorphism. Then, by Lemma~\ref{lem:2}, there is
an isomorphic embedding $L_{\omega }$ from $\bigoplus _{n=1}^{\infty }
\mathcal{F}_{p}(\mathcal{N}_{n})$ into $\bigoplus _{n=1}^{\infty }
\mathcal{F}_{p}(\mathcal{M}_{n})$ given by
\begin{equation*}
(x_{n})_{n=1}^{\infty }\mapsto (L_{n}(x_{n}))_{n=1}^{\infty }.
\end{equation*}
Hence, $\sup _{n} \Vert L_{n}^{-1}\Vert \le \Vert L_{\omega }^{-1}
\Vert <\infty $, which is an absurdity.
\end{proof}

\subsection{Quotients of quasimetric spaces}\label{Sect2.2}
Suppose that $(\mathcal{M},d,0)$ is a pointed $p$-metric space,
$0<p\leq 1$, and that $\mathcal{N}\subset \mathcal{M}$ is a closed
subset of $\mathcal{M}$ with $0\in \mathcal{N}$. Following
\cite{BuragoBook} we put
\begin{equation*}
d_{\mathcal{M}/\mathcal{N}}(x,y):= \min \{d(x,y), (d^{p}(x,
\mathcal{N})+d^{p}(y,\mathcal{N}))^{1/p}\},\quad x,y\in \mathcal{M}.
\end{equation*}
It is clear that $d_{\mathcal{M}/\mathcal{N}}(\cdot , \cdot )$ is
symmetric and that $d_{\mathcal{M}/\mathcal{N}}(x,y)=0$ if and only if
either $x=y$ or $x,y\in \mathcal{N}$. Moreover, it is straightforward
to check that $d_{\mathcal{M}/\mathcal{N}}$ satisfies the $p$-triangle
law. Hence $((\mathcal{M}\setminus \mathcal{N})\cup \{0\},d_{
\mathcal{M}/\mathcal{N}},0)$ is a pointed $p$-metric space, which we
denote by $\mathcal{M}/\mathcal{N}$ and that we call the \emph{quotient
of $\mathcal{M}$ by $\mathcal{N}$}.

Let $Q_{\mathcal{M}/\mathcal{N}}\colon \mathcal{M}\to \mathcal{M}/
\mathcal{N}$ be the \emph{quotient map} given by
\begin{equation*}
Q_{\mathcal{M}/\mathcal{N}}(x)=
\begin{cases}
x
& \text{ if } x\in \mathcal{M}\setminus \mathcal{N},
\\
0
& \text{ if } x\in \mathcal{N}.
\end{cases}
\end{equation*}
We start our study of quotient spaces with some elementary properties.

\begin{Proposition}
\label{Prop:LipMapQuotients}
Suppose that $(\mathcal{M},d,0)$ is a pointed $p$-metric space,
$0<p\leq 1$, and that $\mathcal{N}\subseteq \mathcal{M}$ is a closed
subset of $\mathcal{M}$ with $0\in \mathcal{N}$. Given another
$p$-metric space $X$, the map $T\colon \mathrm{Lip}_{0}(\mathcal{M}/
\mathcal{N},X) \to \mathrm{Lip}_{0}(\mathcal{M},X)$ defined by
$T(f)=f\circ Q_{\mathcal{M}/\mathcal{N}}$ is an isometric embedding with
range
\begin{equation*}
\mathrm{Lip}_{\mathcal{N}}(\mathcal{M},X):=\{f\in \mathrm{Lip}(M,X)
\colon f|_{\mathcal{N}}=0\}.
\end{equation*}
Moreover, if $X$ is $p$-normed, $T$ is linear.
\end{Proposition}

\begin{proof}
It is clear that if $f\in \mathrm{Lip}_{0}(\mathcal{M}/\mathcal{N},X)$,
$f\circ Q_{\mathcal{M}/\mathcal{N}}\in \mathrm{Lip}_{\mathcal{N}}(
\mathcal{M},X)$ and $\operatorname{\mathrm{{Lip}}}(f\circ Q_{
\mathcal{M}/\mathcal{N}})\le \mathrm{Lip}(f)$. Let $g\colon
\mathcal{M}\to X$ be a $C$-Lipschitz map with $g(x)=0$ for all
$x\in \mathcal{N}$. Then, for $x\in \mathcal{M}$,
\begin{equation*}
d(g(x),0)=\inf _{y\in \mathcal{N}} d(g(x), g(y))\le C
\inf _{y\in \mathcal{N}} d(x,y)=Cd(x,\mathcal{N}).
\end{equation*}
Consequently, for all $x$, $y\in \mathcal{M}$,
\begin{align*}
d(g(x),g(y))
&\le \left ( d^{p}(g(x),0)+d^{p}(g(y),0) \right )^{1/p}\\
&\le C (d^{p}(x,\mathcal{N}) + d^{p}(y,\mathcal{N}))^{1/p}.
\end{align*}
Hence $d(g(x),g(y))\le C d_{\mathcal{M}/\mathcal{N}}(x,y)$. So, if
$f\colon \mathcal{M}/\mathcal{N}\to X$ is the unique map satisfying
$f\circ Q_{\mathcal{M}/\mathcal{N}}=g$, $f$ is $C$-Lipschitz.
\end{proof}

Proposition~\ref{Prop:LipMapQuotients} allows us to identify $q$-metric
envelopes of quotients of $p$-metric spaces for $0<p<q\le 1$. We refer
the reader to \cite{AACD2018}*{\S 3} for a precise definition of
metric envelopes, a concept that arises most naturally while
investigating Lipschitz free $p$-spaces over quasimetric spaces.

\begin{Corollary}
\label{cor:MetricEnvelope}Let $0<p\le q\le 1$. Suppose $(\mathcal{M},d,0)$ is a pointed $p$-metric
space and that $\mathcal{N}\subseteq \mathcal{M}$ is a closed subset of
$\mathcal{M}$ with $0\in \mathcal{N}$. Let $\mathcal{N}^{q}$ be the
closure of $J_{q}(\mathcal{N})$ in the $q$-metric envelope $(
\widetilde{\mathcal{M}}^{q},J_{q})$ of $\mathcal{M}$. Then the
$q$-metric envelope of $\mathcal{M}/\mathcal{N}$ is $(
\widetilde{\mathcal{M}}^{q}/\mathcal{N}^{q}, J_{q,\mathcal{N}})$, where
$J_{q,\mathcal{N}}$ is the unique map for which the diagram
\begin{equation}
\label{diam1}
\xymatrix{\ensuremath {\mathcal {M}}\ar[r]^{J_{q}} \ar[d]_{Q_{\ensuremath {\mathcal {M}}/\ensuremath {\mathcal {N}}}} & \widetilde{\ensuremath {\mathcal {M}}^{q}}\ar[d]^{Q_{\widetilde{\ensuremath {\mathcal {M}}^{q}}/{\ensuremath {\mathcal {N}}^{q}}}} \\
\ensuremath {\mathcal {M}}/\ensuremath {\mathcal {N}}\ar[r]^{J_{q,\ensuremath {\mathcal {N}}}}& \widetilde{\ensuremath {\mathcal {M}}^{q}}/ {\ensuremath {\mathcal {N}}^{q}} }
\end{equation}
commutes.
\end{Corollary}
\begin{proof}
Proposition~\ref{Prop:LipMapQuotients} applied to $Q_{
\widetilde{\mathcal{M}^{q}}/{\mathcal{N}^{q}}}\circ J_{q} \in
\mathrm{Lip}_{\mathcal{N}}(\mathcal{M},\widetilde{\mathcal{M}^{q}}/
{\mathcal{N}^{q}})$ gives the existence and uniqueness of the map
$J_{q,\mathcal{N}}$ with $\mathrm{Lip}(J_{q,\mathcal{N}}) = 1$ that
makes the diagram \eqref{diam1} commutative. Let $f\colon
\mathcal{M}/\mathcal{N}\to X$ be a $C$-Lipschitz map into a $q$-metric
space. If we set $f(0)$ as the base point of $X$, applying first the
universal property of $q$-metric envelopes and then that of quotients,
we obtain unique $C$-Lipschitz maps $g$ and $h$ such that the diagram
\begin{equation}
\label{diam2}
\xymatrix{\ensuremath {\mathcal {M}}\ar[r]^{J_{q}} \ar[d]_{Q_{\ensuremath {\mathcal {M}}/\ensuremath {\mathcal {N}}}} & \widetilde{\ensuremath {\mathcal {M}}^{q}} \ar[ddl]^{g}\ar[d]^{Q_{\widetilde{\ensuremath {\mathcal {M}}^{q}}/{\ensuremath {\mathcal {N}}^{q}}}} \\
\ensuremath {\mathcal {M}}/\ensuremath {\mathcal {N}}\ar[d]_f & \widetilde{\ensuremath {\mathcal {M}}^{q}}/ {\ensuremath {\mathcal {N}}^{q}} \ar[dl]^{h} \\
X &
}
\end{equation}
commutes. Using that $Q_{\mathcal{M}/\mathcal{N}}$ is onto we deduce
that \eqref{diam1} and \eqref{diam2} merge in the commutative diagram
\begin{equation}
\label{diam3}
\xymatrix{\ensuremath {\mathcal {M}}\ar[r]^{J_{q}} \ar[d]_{Q_{\ensuremath {\mathcal {M}}/\ensuremath {\mathcal {N}}}} & \widetilde{\ensuremath {\mathcal {M}}^{q}}\ar[d]^{Q_{\widetilde{\ensuremath {\mathcal {M}}^{q}}/{\ensuremath {\mathcal {N}}^{q}}}} \\
\ensuremath {\mathcal {M}}/\ensuremath {\mathcal {N}}\ar[r]^{J_{q,\ensuremath {\mathcal {N}}}} \ar[d]_f & \widetilde{\ensuremath {\mathcal {M}}^{q}}/ {\ensuremath {\mathcal {N}}^{q}} \ar[dl]^{h} \\
X &
}
\end{equation}
The uniqueness of $g$ and $h$ in \eqref{diam2} gives that the map $h$ in \eqref{diam3} is unique.
\end{proof}

\begin{Corollary}
\label{Cor:BanachEnvelopeQuotient}
Let $0<p\le q\le 1$. Suppose that $(\mathcal{M},d,0)$ is a pointed
$p$-metric space and that $\mathcal{N}$ is a closed
subset of $\mathcal{M}$ with $0\in \mathcal{N}$. Then the $q$-Banach
envelope of $\mathcal{F}_{p}(\mathcal{M}/\mathcal{N})$ is $
\mathcal{F}_{q}(\widetilde{\mathcal{M}^{q}}/\mathcal{N}^{q})$.
\end{Corollary}

\begin{proof}
Just combine Corollary~\ref{cor:MetricEnvelope} with
\cite{AACD2018}*{Proposition 4.20}.
\end{proof}

\begin{Corollary}
\label{Cor:DualityQuotient}
Let $0<p\le 1$. Suppose that $(\mathcal{M},d,0)$ is a pointed $p$-metric
space and that $\mathcal{N}$ is a closed subset of
$\mathcal{M}$ with $0\in \mathcal{N}$. Then the dual space of
$\mathcal{F}_{p}(\mathcal{M}/\mathcal{N})$ is $\mathrm{Lip}_{
\mathcal{N}}(\mathcal{M},\mathbb{R})$ under the dual pairing given by
$\langle f,\delta _{\mathcal{M}/\mathcal{N}}(x)\rangle =f(x)$ for all
$f\in \mathrm{Lip}_{\mathcal{N}}(\mathcal{M},\mathbb{R})$ and all
$x\in \mathcal{M}\setminus \mathcal{N}$.
\end{Corollary}

\begin{proof}
Just combine Proposition~\ref{Prop:LipMapQuotients} with
\cite{AACD2018}*{Proposition 4.23}.
\end{proof}

Lipschitz free spaces over quotients yield an alternative description
of some Sobolev spaces. Prior to state and prove this result, we write
down an elementary functional lemma that we will need. We omit the proof.

\begin{Lemma}
\label{lem:F}Let $X$ and $Y$ be Banach spaces, let $X_{0}$ be a dense subspace of $X$, 
and let $T\colon X_{0} \to Y$ be a linear map. Suppose there is an
isomorphism $S\colon Y^{*}\to X^{*}$ such that $S(y^{*})(x)=y^{*}(T(x))$
for every $y^{*}\in Y^{*}$ and every $x\in X_{0}$. Then $T$ extends to an
isomorphism from $X$ onto $Y$. Moreover, if $S$ is an isometry, so is
$T$.
\end{Lemma}

Given a bounded open  subset $U$ of $\mathbb{R}^{d}$, following
\cite{M14}*{Definitions 2.2 and 2.3}, we define the Sobolev space $W^{-1,1}(U)$ as the closure of
$\mathcal{C}_{0}(U)^{\ast }$ with respect to the dual norm in $\mathrm{Lip}_{U^{c}}(\mathbb{R}^{d})^{\ast }$.

\begin{Theorem}
\label{thm:sobolev}Let  $U$ be a bounded open 
set in $\mathbb{R}^{d}$,  $d\in \mathbb{N}$. Then $\mathcal{F}(\mathbb{R}^{d}/U^{c})\simeq W^{-1,1}(U)$
isometrically.
\end{Theorem}
\begin{proof}
Let
\begin{equation*}
\langle \cdot , \cdot \rangle _{0}\colon \mathrm{Lip}_{U^{c}}(
\mathbb{R}^{d})\times \mathcal{F}(\mathbb{R}^{d}/U^{c}) \to
\mathbb{R}
\end{equation*}
be the dual pairing between $\mathrm{Lip}_{U^{c}}(\mathbb{R}^{d})$ and
$\mathcal{F}(\mathbb{R}^{d}/U^{c})$ provided by
Corollary~\ref{Cor:DualityQuotient}, and let $\delta (x)$ denote the
Dirac measure on $U$ at the point $x\in U$. It is known (see
\cite{M14}*{Proposition 8.7}) that the dual of the Sobolev space
$W^{-1,1}(U)$ is isometric to $\mathrm{Lip}_{U^{c}}(\mathbb{R}^{d})$ and
that, if $\langle \cdot , \cdot \rangle _{1}$ denotes the associated dual
pairing, we have $\langle f, \delta (x)\rangle _{1}=f(x)$ for every
$f\in \mathrm{Lip}_{U^{c}}(\mathbb{R}^{d})$ and every $x\in U$.
Therefore, if we define the linear map
\begin{equation*}
T\colon \mathcal{P}(\mathbb{R}^{d}/U^{c})\to W^{-1,1}(U),\quad
\delta _{\mathbb{R}^{d}/U^{c}}(x)\mapsto \delta (x),
\end{equation*}
we have $\langle f, \mu \rangle _{0}=\langle f, T(\mu )\rangle _{1}$. By
Lemma~\ref{lem:F}, $T$ extends to an isometry from $\mathcal{F}(
\mathbb{R}^{d}/U^{c})$ onto $W^{-1,1}(U)$.
\end{proof}

\begin{Remark}
After consulting with experts on the field, we found out that the
literature is not unified in regards to the right definition of the Sobolev
space $W^{-1,1}$. We used the definition from \cite{M14} because, in
this case, $W^{-1,1}$ is a canonical predual of $\mathrm{Lip}_{U^{c}}(
\mathbb{R}^{d},\mathbb{R})$. In general, the proof above shows that any
``natural predual'' of the space $\mathrm{Lip}_{U^{c}}(\mathbb{R}^{d},
\mathbb{R})$ is isometric to $\mathcal{F}(\mathbb{R}^{d}/U^{c})$.
Notice that by a recent result of Weaver, Lipschitz free spaces over bounded metric spaces have strongly unique preduals
(see \cite{Weaver2019}). Hence, when combined with Proposition~\ref{Prop:LipMapQuotients}, this gives an alternative proof to Theorem~\ref{thm:sobolev}.
\end{Remark}

\subsection{Lipschitz retractions}\label{Sect2.3}
It is known that if $\mathcal{N}$ is a Lipschitz retract of
$\mathcal{M}$ then the space $\mathcal{F}_{p}(\mathcal{N})$ is a
complemented subspace of $\mathcal{F}_{p}(\mathcal{M})$ via the
canonical linear map from $\mathcal{F}_{p}(\mathcal{N})$ into
$\mathcal{F}_{p}(\mathcal{M})$ (see \cite{AACD2018}*{Lemma 4.19}). The
following result identifies a complement of $\mathcal{F}_{p}(
\mathcal{N})$ in $\mathcal{F}_{p}(\mathcal{M})$.

\begin{Theorem}
\label{thm:retr}Let $(\mathcal{M},d,0)$ be a pointed $p$-metric space, $0<p\leq 1$, and
$\mathcal{N}\subseteq \mathcal{M}$ be a Lipschitz retract. Then
\begin{equation*}
\mathcal{F}_{p}(\mathcal{M})\simeq \mathcal{F}_{p}(\mathcal{N})\oplus
\mathcal{F}_{p}(\mathcal{M}/\mathcal{N}).
\end{equation*}
\end{Theorem}
\begin{proof}
Let $r\colon \mathcal{M}\to \mathcal{M}$ be a Lipschitz retraction with
$r(\mathcal{M})=\mathcal{N}$ and set $L=\mathrm{Lip}(r)$. Let
$f\colon \mathcal{M}\rightarrow \mathcal{F}_{p}(\mathcal{N}\maltese
\mathcal{M}/\mathcal{N})$ be defined as
\begin{equation*}
f(x)=
\begin{cases}
\delta _{\maltese }(r(x),0) +\delta _{\maltese }(0,x)
& \text{ if } x
\in \mathcal{M}\setminus \mathcal{N},
\\
\delta _{\maltese }(x,0)
& \text{ if } x\in \mathcal{N},
\end{cases}
\end{equation*}
where $\delta _{\maltese }$ denotes the natural isometric embedding of
$\mathcal{N}\maltese \mathcal{M}/\mathcal{N}$ into $\mathcal{F}_{p}(
\mathcal{N}\maltese \mathcal{M}/\mathcal{N})$. Since $r(x)=x$ for
$x\in \mathcal{N}$ we obtain
\begin{equation*}
\Vert f(x)-f(y)\Vert ^{p}\le
\begin{cases}
A^{p}(x,y)+ B^{p}(x,y)
& \text{ if } x,y \in \mathcal{M}\setminus
\mathcal{N},
\\
A^{p}(x,y)
& \text{ if } x,y \in \mathcal{N},
\\
A^{p}(x,y) +B^{p}(x,0)
& \text{ if } x \in \mathcal{M}\setminus
\mathcal{N}\text{ and } y\in \mathcal{N},
\\
\end{cases}
\end{equation*}
where $A(x,y)=d(r(x),r(y))$ and
$B(x,y)=d_{\mathcal{M}/\mathcal{N}}(x,y)$. Since, for all $x$,
$y\in \mathcal{M}$,
\begin{equation*}
A(x,y)\le L \, d(x,y),\quad B(x,y)\le d(x,y),
\end{equation*}
and, if $y\in \mathcal{N}$,
\begin{equation*}
B(x,0)\le d(x,\mathcal{N})\le d(x,y),
\end{equation*}
the function $f$ is $(L^{p}+1)^{1/p}$-Lipschitz. So, by
\cite{AACD2018}*{Theorem 4.5}, there is a linear operator $T\colon
\mathcal{F}_{p}(\mathcal{M})\rightarrow \mathcal{F}_{p}(\mathcal{N}
\maltese \mathcal{M}/\mathcal{N})$ such that
$\Vert T\Vert \le (L^{p}+1)^{1/p}$ and $T\circ \delta _{\mathcal{M}}=f$.

Conversely, we define a map $g\colon \mathcal{N}\maltese \mathcal{M}/
\mathcal{N}\rightarrow \mathcal{F}_{p}(\mathcal{M})$ as
\begin{equation*}
g(x_{1},x_{2})=
\begin{cases}
\delta _{\mathcal{M}}(x_{1})
& \text{ if } x_{2}=0,
\\
\delta _{\mathcal{M}}(x_{2})-\delta _{\mathcal{M}}(r(x_{2}))
& \text{ if
} x_{2}\neq0.
\end{cases}
\end{equation*}
Let $(x_{1},x_{2})$, $(y_{1},y_{2})\in \mathcal{N}\maltese
\mathcal{M}/\mathcal{N}$. Since $y_{1}=0$ if $y_{2}\neq0$, we deduce
that $\Vert g(x_{1},x_{2}) -g(y_{1},y_{2})\Vert ^{p}$ is bounded above
by
\begin{equation*}
\begin{cases}
d^{p}(x_{1},y_{1})
& \text{ if } x_{2}=y_{2}=0,
\\
d^{p}(x_{2},y_{2})+A^{p}(x_{2},y_{2})
& \text{ if } x_{2},y_{2}
\neq0,
\\
d^{p}(x_{1},y_{1})+d^{p}(x_{2},y_{2}) + A^{p}(x_{2},y_{2})
& \text{ if
} x_{2}=0, y_{2}\neq0.
\end{cases}
\end{equation*}
Thus, for all $(x_{1},y_{1})$ and $(x_{2},y_{2})\in \mathcal{N}\maltese
\mathcal{M}/\mathcal{N}$ we have
\begin{align*}
\Vert g(x_{1},x_{2}) -g(y_{1},y_{2})\Vert ^{p}
&\le (1+L^{p}) \big (d
^{p}(x_{1},y_{1}) + d^{p}(x_{2},y_{2})\big )
\\
&\le (1+L^{p}) \big (d^{p}(x_{1},y_{1}) + d_{\mathcal{M}/\mathcal{N}}
^{p}(x_{2},y_{2})\big ),
\end{align*}
and so there is a linear map $S\colon \mathcal{F}_{p}(\mathcal{N}\maltese
\mathcal{M}/\mathcal{N})\rightarrow \mathcal{F}_{p}(\mathcal{M})$ such
that $S\circ \delta _{\maltese }=g$ and
$\Vert S\Vert \le (1+L^{p})^{1/p}$.

Finally, it is easy to verify that $S\circ T\circ \delta _{\mathcal{M}}=
\delta _{\mathcal{M}}$ and $S\circ T\circ
\delta _{\mathcal{N}\maltese \mathcal{M}/\mathcal{N}} =
\delta _{\mathcal{N}\maltese \mathcal{M}/\mathcal{N}}$, hence
$S\circ T=\operatorname{\mathrm{{Id}}}_{\mathcal{F}_{p}(\mathcal{M})}$
and $S\circ T=\operatorname{\mathrm{{Id}}}_{\mathcal{F}_{p}(
\mathcal{N}\maltese \mathcal{M}/\mathcal{N})}$.

Summarizing and appealing to Lemma~\ref{lem:2} we obtain
\begin{equation*}
\mathcal{F}_{p}(\mathcal{M})\simeq \mathcal{F}_{p}(\mathcal{N}\maltese
\mathcal{M}/\mathcal{N}) \simeq \mathcal{F}_{p}(\mathcal{N})\oplus
\mathcal{F}_{p}(\mathcal{M}/\mathcal{N}).\qedhere
\end{equation*}
\end{proof}

Let us mention an application that  will be used later. Recall that a
subset $\mathcal{N}$ of a quasimetric space $(\mathcal{M},d)$ is said to
be a \emph{net} if
\begin{equation*}
\inf \{d(x,y) \colon x,y\in \mathcal{N}, x\not =y\}>0\quad \text{and}
\quad \sup _{x\in \mathcal{M}} d(x,\mathcal{N})<\infty .
\end{equation*}

\begin{Corollary}
\label{cor:nets}
Suppose $\mathcal{M}$ is a uniformly separated $p$-metric space,
$0<p\leq 1$, and that $\mathcal{N}\subseteq \mathcal{M}$ is a net. Then
\begin{equation*}
\mathcal{F}_{p}(\mathcal{M})\simeq \mathcal{F}_{p}(\mathcal{N})\oplus
\ell _{p}(|\mathcal{M}\setminus \mathcal{N}|).
\end{equation*}
\end{Corollary}
\begin{proof}
Given $x\in \mathcal{M}\setminus \mathcal{N}$, pick $z(x)\in
\mathcal{N}$ such that $d(x,z(x))\le 2 d(x,\mathcal{N})$. The map
\begin{equation*}
x\mapsto
\begin{cases}
x
& \text{ if } x\in \mathcal{N},
\\
z(x)
& \text{ if } x\notin \mathcal{N},
\end{cases}
\end{equation*}
is a Lipschitz retract of $\mathcal{M}$ onto $\mathcal{N}$. Indeed, if
$x$ and $y$ are different points in $\mathcal{M}\setminus \mathcal{N}$,
\begin{align*}
d^{p}(z(x),z(y))
&\le d^{p}(x,y) + 2^{p} (d^{p}(x,\mathcal{N})+d^{p}(y,
\mathcal{N}))\\
& \le \left (1+2^{p+1} \frac{R^{p}}{\varepsilon ^{p}}\right )d
^{p}(x,y),
\end{align*}
where $R=\sup _{x\in \mathcal{M}} d(x,\mathcal{N})$ and $\varepsilon =
\inf \{d(x,y) \colon x,y\in \mathcal{M}, x\not =y\}>0$. By
Theorem~\ref{thm:retr},
\begin{equation*}
\mathcal{F}_{p}(\mathcal{M})\simeq \mathcal{F}_{p}(\mathcal{N})\oplus
\mathcal{F}_{p}(\mathcal{M}/\mathcal{N}).
\end{equation*}
Since $\mathcal{M}/\mathcal{N}$ is uniformly separated and bounded,
\cite{AACD2018}*{Theorem 4.14} yields $\mathcal{F}_{p}(\mathcal{M}/
\mathcal{N})\simeq \ell _{p}(|\mathcal{M}\setminus \mathcal{N}|)$.
\end{proof}

\section{Complementability of $\ell _{p}$ in $\mathcal{F}_{p}(\mathcal{M})$ for $0<p\le 1$ and Lipschitz free $p$-spaces over nets}\label{Section:Comp}
\noindent

Our main result on complementability of $\ell _{p}$ in Lipschitz free
$p$-spaces is the following generalization of \cite{CDW2016}*{Theorem~1.1}
and \cite{HN17}*{Proposition 3}. We will obtain
Theorem~\ref{thm:mainComplementedEllp} from
Theorems~\ref{thm:discretePoints} and \ref{thm:ellPComplementedMetricCase} below. Given a topological space
$\mathcal{M}$, $\operatorname{dens} \mathcal{M}$ will
denote the density character of $\mathcal{M}$, i.e., the minimal
cardinality of a dense subset of $\mathcal{M}$.

\begin{Theorem}
\label{thm:mainComplementedEllp}
Let $p\in (0,1]$. Suppose that $(\mathcal{M},d)$ is either
\begin{enumerate}[(a)]\item[(a)] a metric space, or
\item[(b)] a $p$-metric space containing $
\operatorname{dens} \mathcal{M}$-many isolated points.
\end{enumerate}
Then for every $C>2^{1/p}$, $\ell _{p}(\operatorname{dens} \mathcal{M})$ is $C$-complemented in $\mathcal{F}_{p}(
\mathcal{M})$.
\end{Theorem}

We do not attempt to achieve an optimal quantitative estimate here, but
what we consider interesting is that it does not depend on the space
$\mathcal{M}$. Let us note that this seems to be a new result even for
$p=1$ and nonseparable metric spaces. For separable metric spaces and
$p=1$, quantitative estimates other than $C=1$ are not important because,
by \cite{DRT}, whenever a Banach space $X$ has a complemented subspace
isomorphic to $\ell _{1}$, then for each $\varepsilon >0$ it has a
$(1+\varepsilon )$-complemented subspace $(1+\varepsilon )$--isomorphic
to $\ell _{1}$.

Sometimes it is even possible to have a more precise information about
the complemented copy of $\ell _{p}$ in the space $\mathcal{F}_{p}(
\mathcal{M})$ at the cost of losing the quantitative estimate. For
instance, if $\mathcal{M}$ is a uniformly separated infinite $p$-metric
space containing a net $\mathcal{N}$ with $|\mathcal{M}\setminus
\mathcal{N}|=|\mathcal{M}|$, the proof of Corollary~\ref{cor:nets}
allows us to identify the complemented copy of $\ell _{p}(\mathcal{M})$
inside $\mathcal{F}_{p}(\mathcal{M})$ as a Lipschitz free $p$-space over
a quotient space. In the following result, which also seems to be new even
for the case of $p=1$, we identify the aforementioned copy of
$\ell _{p}$ inside $\mathcal{F}_{p}(\mathcal{M})$ as a Lipschitz free
$p$-space on a subset of $\mathcal{M}$.

\begin{Theorem}
\label{thm:mainComplementedEllp2}
Let $p\in (0,1]$. Suppose that $(\mathcal{M},d)$ is either
\begin{enumerate}[(a)]\item[(a)] an infinite metric space, or
\item[(b)] a uniformly separated uncountable $p$-metric space.
\end{enumerate}
Then there exists $\mathcal{N}\subset \mathcal{M}$ such that
\begin{enumerate}[(iii)]\item[(i)] $\mathcal{F}_{p}(\mathcal{N})\simeq \ell _{p}$,
\item[(ii)] $L_{\jmath }\colon \mathcal{F}_{p}(\mathcal{N})\to
\mathcal{F}_{p}(\mathcal{M})$ is an isomorphic embedding, and
\item[(iii)] $L_{\jmath }(\mathcal{F}_{p}(\mathcal{N}))$ is complemented
in $\mathcal{F}_{p}(\mathcal{M})$.
\end{enumerate}
\end{Theorem}

Before tackling the proof of Theorems~\ref{thm:mainComplementedEllp} and \ref{thm:mainComplementedEllp2}, let us highlight some interesting
applications.

\begin{Corollary}
\label{cor:addingRd}Let $0<p\le 1$. There is a constant $C$ such that for every
$n\in \mathbb{N}$ and every infinite metric space $\mathcal{M}$,
\begin{equation*}
\mathcal{F}_{p}(\mathcal{M}) \simeq _{C} \mathcal{F}_{p}(\mathcal{M})
\oplus \ell _{p}^{n} \simeq _{C} \mathcal{F}_{p}(\mathcal{M})\oplus \ell
_{p}.
\end{equation*}
\end{Corollary}

\begin{Proposition}
\label{prop:dropapoint}Let $(\mathcal{M},d)$ be an infinite metric space and $0<p\le 1$. Then
$\mathcal{F}_{p}(\mathcal{M})\simeq \mathcal{F}_{p}(\mathcal{M}\setminus
\{x_{0}\})$ for every $x_{0}\in \mathcal{M}$.
\end{Proposition}
\begin{proof}
If $r=\inf _{x\in \mathcal{M}\setminus \{x_{0}\}} d(x_{0},x)=0$ the
result follows from \cite{AACD2018}*{Proposition 4.17}. Assume that
$r>0$ and pick an arbitrary point $0\in \mathcal{M}\setminus \{x_{0}
\}$. If $x\in \mathcal{M}\setminus \{x_{0}\}$ we have
\begin{equation*}
d^{p}(0,x_{0})+d^{p}(0,x)\le 2 d^{p}(0,x_{0}) + d^{p}(x_{0},x)\le
\left ( 1 + 2 \frac{ d^{p}(0,x_{0}) }{r^{p}} \right ) d^{p}(x_{0},x).
\end{equation*}
Then, by Lemma~\ref{lem:isoWithELLPSum} and
Corollary~\ref{cor:addingRd},
\begin{align*}
\mathcal{F}_{p}(\mathcal{M})
& \simeq \mathcal{F}_{p}(\mathcal{M}
\setminus \{x_{0}\}) \oplus \mathcal{F}_{p}(\{0,x_{0}\})
\\
& \simeq \mathcal{F}_{p}(\mathcal{M}\setminus \{x_{0}\}) \oplus
\mathbb{R}
\\
& \simeq \mathcal{F}_{p}(\mathcal{M}\setminus \{x_{0}\}).\qedhere
\end{align*}
\end{proof}

Further, we obtain the following extensions of \cite{HN17}*{Theorem
4 and Proposition 5} to the whole range of values of $p\in (0,1]$.
Recall that there are examples of nets in $\mathbb{R}^{2}$ which are not
bi-Lipschitz equivalent (see e.g.~\cite{BK98}, \cite{M98} or
\cite{BenLin2000}*{p. 242}).

\begin{Theorem}
\label{thm:nets}
Let $0<p\leq 1$. Suppose that $\mathcal{M}$ is a uniformly separated
$p$-metric space and that $\mathcal{N}\subseteq \mathcal{M}$ is a net
such that $|\mathcal{N}|=|\mathcal{M}|$. Then $\mathcal{F}_{p}(
\mathcal{M})\simeq \mathcal{F}_{p}(\mathcal{N})$.
\end{Theorem}
\begin{proof}
By Theorem~\ref{thm:mainComplementedEllp}, we have $\mathcal{F}_{p}(
\mathcal{N})\simeq X \oplus \ell _{p}(|\mathcal{N}|)$ for some $p$-Banach
space $X$. Then, appealing to Theorem~\ref{cor:nets},
\begin{align*}
\mathcal{F}_{p}(\mathcal{M})
&\simeq X\oplus \ell _{p}(|\mathcal{N}|)
\oplus \ell _{p}(|\mathcal{M}\setminus \mathcal{N}|)
\\
&\simeq X\oplus \ell _{p}(|\mathcal{M}|)
\\
&=X\oplus \ell _{p}(|\mathcal{N}|)
\\
&\simeq \mathcal{F}_{p}(\mathcal{N}).\qedhere
\end{align*}
\end{proof}

\begin{Proposition}
\label{prop:nets}Let $\mathcal{M}$ be a $p$-metric space, $0<p\leq 1$. Suppose
$\mathcal{N}_{1}$ and $\mathcal{N}_{2}$ are nets in $\mathcal{M}$ of the
same cardinality, $\operatorname{dens} \mathcal{M}$.
Then $\mathcal{F}_{p}(\mathcal{N}_{1})\simeq \mathcal{F}_{p}(
\mathcal{N}_{2})$.
\end{Proposition}
\begin{proof}
The proof follows from Theorem~\ref{thm:nets} exactly in the same way
as \cite{HN17}*{Proposition~5} follows from \cite{HN17}*{Theorem 4}.
\end{proof}

Before presenting the results on which
Theorems~\ref{thm:mainComplementedEllp} and \ref{thm:mainComplementedEllp2} are based, let us note that for $p=1$
and for separable spaces the result has a quantitative improvement in
\cite{CuthJohanis2017}, where it is proved that $\ell _{1}$ is
isometric to a $1$-complemented subspace of $\mathcal{F}(\mathcal{M})$
whenever $\mathcal{M}$ has an accumulation point or contains an infinite
ultrametric space. However, by the recent work \cite{OO19}, there
is a metric space $\mathcal{M}$ such that $\ell _{1}$ does not
isometrically embed into $\mathcal{F}(\mathcal{M})$. These advances
suggest several natural areas for further research,
see, e.g., Questions~\ref{qu:1complemented} and \ref{qu:1embeded}.

Finally, in what remains of this section we provide results that imply
Theorem~\ref{thm:mainComplementedEllp} and
Theorem~\ref{thm:mainComplementedEllp2}. The arguments in our proofs are
inspired by \cite{CDW2016}, \cite{CuthJohanis2017} and
\cite{HN17}.

Given a map $f\colon X \to Y$, where $X$ is a set and $Y$ is a vector
space, we shall denote the set $ f^{-1}(Y\setminus \{0\})$ by
$\operatorname{\mathrm{{supp}}}_{0}(f)$. 

\begin{Lemma}
\label{lemma:condition2}
Let $(\mathcal{M},d)$ be a $p$-metric space, $C>0$, $(x_{\gamma })_{
\gamma \in \Gamma }$ and $(y_{\gamma })_{\gamma \in \Gamma }$ be
sequences in $\mathcal{M}$, and $(f_{\gamma })_{\gamma \in \Gamma }$ be
a sequence of $C$-Lipschitz maps from $(\mathcal{M},d)$ into
$\mathbb{R}$ such that $(\operatorname{\mathrm{{supp}}}_{0}(f_{\gamma
}))_{\gamma \in \Gamma }$ is a pairwise disjoint sequence, $f_{\gamma
}(x_{\gamma })\neq 0$ for every $\gamma \in \Gamma $, $f_{\gamma _{1}}(x
_{\gamma _{2}}) = 0$ for every $\gamma _{1}$, $\gamma _{2}\in \Gamma $ with
$\gamma _{1}\neq\gamma _{2}$, and $f_{\gamma _{1}}(y_{\gamma _{2}}) = 0$
for every $\gamma _{1}$, $\gamma _{2}\in \Gamma $. Finally, suppose there
exists $t>0$ with
\begin{equation*}
\frac{f_{\gamma }(x_{\gamma })}{d(x_{\gamma },y_{\gamma })}\geq
\frac{1}{t}, \quad \gamma \in \Gamma .
\end{equation*}
Then $\ell _{p}(\Gamma )$ is $2^{1/p}Ct$-complemented in $\mathcal{F}
_{p}(\mathcal{M})$.

Moreover, if  $y_{\gamma } = 0$ for every $\gamma \in \Gamma $,
and set $\mathcal{N}= \{x_{\gamma }\colon \gamma \in \Gamma
\}\cup \{0\}$, we have:
\begin{enumerate}[(iii)]\item[(i)] $\mathcal{F}_{p}(\mathcal{N})\simeq \ell _{p}(\Gamma )$;
\item[(ii)] $L_{\jmath }\colon \mathcal{F}_{p}(\mathcal{N})\to
\mathcal{F}_{p}(\mathcal{M})$ is an isomorphic embedding; and
\item[(iii)] $L_{\jmath }(\mathcal{F}_{p}(\mathcal{N}))$ is complemented
in $\mathcal{F}_{p}(\mathcal{M})$.
\end{enumerate}
\end{Lemma}

\begin{proof}
For each $\gamma \in \Gamma $ put
\begin{equation*}
b_{\gamma }= \frac{\delta _{\mathcal{M}}(x_{\gamma })- \delta _{
\mathcal{M}}(y_{\gamma })}{d(x_{\gamma },y_{\gamma })}\in \mathcal{F}
_{p}(\mathcal{M}).
\end{equation*}
Since $\Vert b_{\gamma }\Vert _{\mathcal{F}_{p}(\mathcal{M})}=1$, there
is a norm-one linear operator $S\colon \ell _{p}(\Gamma ) \to
\mathcal{F}_{p}(\mathcal{M})$ such that $S(\mathbf{e}_{\gamma })=b
_{\gamma }$ for all $\gamma \in \Gamma $.

Define
\begin{equation*}
f\colon \mathcal{M}\to \ell _{p}(\Gamma ), \quad x \mapsto
\sum _{\gamma \in \Gamma } \frac{d(x_{\gamma },y_{\gamma })}{f_{\gamma
}(x_{\gamma })} f_{\gamma }(x)\, \mathbf{e}_{\gamma }.
\end{equation*}
The map $f$ is $2^{1/p} C t$-Lipschitz. Indeed, for every $x$,
$y\in \mathcal{M}$ there are $\gamma _{1}, \gamma _{2}\in \Gamma $ such
that $f_{\gamma }(x)=f_{\gamma }(y)=0$ if $\gamma \notin \{\gamma _{1},
\gamma _{2}\}$. Hence,
\begin{equation*}
\|f(x)-f(y)\|
\leq t \left (\sum _{j=1}^{2} |f_{\gamma _{j}}(x)-f_{
\gamma _{j}}(y)|^{p} \right )^{1/p}
\leq {2^{1/p}C}{t}d(x,y),
\end{equation*}
and so there is a bounded linear  map $P\colon \mathcal{F}_{p}(
\mathcal{M})\to \ell _{p}(\Gamma )$ satisfying $P\circ \delta _{
\mathcal{M}}=f$ with $\|P\| \leq 2^{1/p}Ct$. We have
\begin{equation*}
P(b_{\gamma })=\frac{P(\delta _{\mathcal{M}}(x_{\gamma })) - P(
\delta _{\mathcal{M}}(y_{\gamma }))}{d(x_{\gamma },y_{\gamma })}
=\frac{f(x
_{\gamma }) - f(y_{\gamma })}{d(x_{\gamma },y_{\gamma })}
=\mathbf{e}
_{\gamma },
\end{equation*}
so that $P\circ S=\operatorname{\mathrm{{Id}}}_{\ell _{p}(\Gamma )}$.

For the ``Moreover'' part, we consider the norm-one linear operator
$S':\ell _{p}(\Gamma )\to \mathcal{F}_{p}(\mathcal{N})$ given by
\begin{equation*}
S'(\mathbf{e}_{\gamma })=\frac{\delta _{\mathcal{N}}(x_{\gamma _{1}})}{d(x
_{\gamma_{1}},0)},\quad \gamma \in \Gamma .
\end{equation*}
Then $S = L_{\jmath }\circ S'$, which combined with the above yields $P\circ L
_{\jmath }\circ S' = \operatorname{\mathrm{{Id}}}_{\ell _{p}(\Gamma )}$.
In particular, $L_{\jmath }\circ S'$ is an isomorphism between
$\ell _{p}(\Gamma )$ and $\mathcal{F}_{p}(\mathcal{N})$, so that
$L_{\jmath }$ is an isomorphism as well.
\end{proof}

In general, there is no tool for building nontrivial Lipschitz maps from
a quasimetric space into the real line. In fact, there are quasimetric
spaces $\mathcal{M}$, such as $(\mathbb{R}, |\cdot |^{1/p})$, for which
$\mathrm{Lip}_{0}(\mathcal{M},\mathbb{R})=\{0\}$ (see e.g.
\cite{Albiac2008}*{Lemma 2.7}). The first case when the situation is
quite different is when the quasimetric space contains isolated points.
Indeed, in that case, as we see next, $\ell _{p}$ embeds complementably
in $\mathcal{F}_{p}(\mathcal{M})$ and so $\ell _{\infty }$ must embed
complementably in $\mathrm{Lip}_{0}(\mathcal{M},\mathbb{R})=
\mathcal{F}_{p}(\mathcal{M})^{\ast }$!

\begin{Theorem}
\label{thm:discretePoints}
Suppose $(\mathcal{M},d)$ is a $p$-metric space and let $\kappa $ be the
cardinality of the set $\{x\colon d(x,\mathcal{M}\setminus \{x\})>0\}$
of isolated points of $\mathcal{M}$. Then for every $C>2^{1/p}$, the
space $\ell _{p}(\kappa )$ is $C$-complemented in $\mathcal{F}_{p}(
\mathcal{M})$.
\end{Theorem}
\begin{proof}
Let us enumerate the set $\{x\colon d(x,\mathcal{M}\setminus \{x\})>0
\}$ as $\{x_{i}\colon i<\kappa \}$. Now we consider the map $f_{i}:=d(x
_{i},\mathcal{M}\setminus \{x_{i}\})\cdot \chi _{\{x_{i}\}}$ and, for
$t>1$ fixed, we let $y_{i}\in \mathcal{M}$ be an arbitrary point with
$d(x_{i},y_{i})\leq td(x_{i},\mathcal{M}\setminus \{x_{i}\})$. Each
$f_{i}$ is 1-Lipschitz and
\begin{equation*}
\frac{f_{i}(x_{i})}{d(x_{i},y_{i})} \geq \frac{1}{t}.
\end{equation*}
Hence, by Lemma~\ref{lemma:condition2}, $\ell _{p}(\kappa)$ is
$2^{1/p}t$-complemented in $\mathcal{F}_{p}(\mathcal{M})$. Since $t>1$
was arbitrary, this finishes the proof.
\end{proof}

Another case when non-trivial real-valued Lipschitz functions are available is when they are defined on metric spaces.
Indeed, if $(\mathcal{M},d)$ is a metric space then the map
$x\mapsto d(x,y)$ is $1$-Lipschitz for every $y\in \mathcal{M}$. This simple
fact is one the foundations of the following result.

\begin{Theorem}
\label{thm:ellPComplementedMetricCase}
Let $(\mathcal{M},d)$ be an infinite metric space and let $p\in (0,1]$. Then
for every $C>2^{1/p}$, the space $\ell _{p}(
\operatorname{dens} \mathcal{M})$ is $C$-complemented
in $\mathcal{F}_{p}(\mathcal{M})$.
\end{Theorem}
\begin{proof}
It is well-known and not very difficult to prove (for a reference see
e.g. \cite{handbookCardinalFunctions}*{Theorem 8.1}) that in any
metric space $\mathcal{M}$ we may find $(\operatorname{dens} \mathcal{M})$-many disjoint balls. By
Theorem~\ref{thm:discretePoints}, we may assume that there do not exist
$(\operatorname{dens} \mathcal{M})$-many isolated
points in $\mathcal{M}$. So we may pick non-isolated points
$(x_{i})_{i<\operatorname{dens} \mathcal{M}}$ in
$\mathcal{M}$ and positive numbers $(r_{i})_{i<
\operatorname{dens} \mathcal{M}}$ such that the balls from the set
$\{B(x_{i},r_{i})\colon i<\operatorname{dens}
\mathcal{M}\}$ are pairwise disjoint. For each $i<\kappa $ we pick
$y_{i}\in B(x_{i},r_{i})\setminus \{x_{i}\}$ and define the
$1$-Lipschitz function $f_{i}\colon \mathcal{M}\to \mathbb{R}$ by
\begin{equation*}
f_{i}(x)=\max \{d(x_{i},y_{i})-d(x,x_{i}),0\}, \quad x\in \mathcal{M}.
\end{equation*}
Then $\operatorname{\mathrm{{supp}}}_{0}f_{i} = B(x_{i},d(x_{i},y_{i}))
\subset B(x_{i},r_{i})$. Since these sets are pairwise disjoint we
obviously have $f_{i}( y_{j})=0$ for $i$, $j <
\operatorname{dens} \mathcal{M}$. Finally, we have
$f_{i}(x_{i}) = d(x_{i},y_{i})$ and so the assumptions of
Lemma~\ref{lemma:condition2} are satisfied with $t=1$. Thus we obtain
that the space $\ell _{p}(\operatorname{dens}
\mathcal{M})$ is even $2^{1/p}$-complemented in $\mathcal{F}_{p}(
\mathcal{M})$.
\end{proof}

\begin{proof}[Proof of Theorem~\ref{thm:mainComplementedEllp}]
Just combine Theorems~\ref{thm:discretePoints} and \ref{thm:ellPComplementedMetricCase}.
\end{proof}

Uniformly separated $p$-metric spaces also admit non-trivial real-val\-ued Lipschitz
functions. Indeed, if $(\mathcal{M},d)$ is $r$-separated, then for every
$y\in \mathcal{M}$ the map
$x\mapsto d^{p}(x,y)$ is Lipschitz, with constant $r^{p-1}$.

\begin{Lemma}
\label{lem:condition3}
Let $(\mathcal{M},d)$ be a pointed $p$-metric space, $p\in (0,1]$.
Suppose there are sequences $(x_{\gamma })_{\gamma \in \Gamma }$ in
$\mathcal{M}$ and $(r_{\gamma })_{\gamma \in \Gamma }$ in $(0,+\infty
)$ such that 
\begin{equation}
\label{cond4:ellPComplemented}
\frac{r_{\gamma }}{d(x_{\gamma },0)}\geq \frac{1}{t},\quad \gamma
\in \Gamma,
\end{equation}
for some $t>0$. With the convention that $r_{0}=0$ and $x_{0}=0$, assume further that:
\begin{enumerate}[(a)]\item[(a)] either $d$ is a metric on $\mathcal{M}$ and
\begin{equation}
\label{cond3:ellPComplemented}
d(x_{\gamma _{1}},x_{\gamma _{2}})\geq r_{\gamma _{1}} + r_{\gamma _{2}},
\quad \gamma _{1},\gamma _{2}\in \Gamma \cup \{0\},\, \gamma _{1}\neq
\gamma _{2},
\end{equation}\item[(b)] or $(\mathcal{M},d)$ is uniformly separated and
\begin{equation}
\label{cond5:ellPComplemented}
d^{p}(x_{\gamma _{1}},x_{\gamma _{2}})\geq r_{\gamma _{1}} + r_{\gamma
_{2}}, \quad \gamma _{1},\gamma _{2}\in \Gamma \cup \{0\},\, \gamma _{1}
\neq \gamma _{2}.
\end{equation}
\end{enumerate}

Then, if we set $\mathcal{N}= \{x_{\gamma }\colon \gamma \in \Gamma
\}\cup \{0\}$, we have:
\begin{enumerate}[(iii)]\item[(i)] $\mathcal{F}_{p}(\mathcal{N})\simeq \ell _{p}(\Gamma )$;
\item[(ii)] $L_{\jmath }\colon \mathcal{F}_{p}(\mathcal{N})\to
\mathcal{F}_{p}(\mathcal{M})$ is an isomorphic embedding; and
\item[(iii)] $L_{\jmath }(\mathcal{F}_{p}(\mathcal{N}))$ is complemented
in $\mathcal{F}_{p}(\mathcal{M})$.
\end{enumerate}
\end{Lemma}
\begin{proof} If (a) holds,
we consider functions $f_{\gamma }:\mathcal{M}\to \mathbb{R}$ given by
\begin{equation*}
f_{\gamma }(x)=\max \{r_{\gamma }-d(x,x_{\gamma }),0\}, \quad x\in
\mathcal{M},
\end{equation*}
whereas if (b) holds, we consider  maps $f_{\gamma }:\mathcal{M}\to \mathbb{R}$  defined by
\begin{equation*}
f_{\gamma }(x)=\max \{r_{\gamma }-d^{p}(x,x_{\gamma }),0\}, \quad x
\in \mathcal{M}.
\end{equation*}
Those functions are $1$-Lipschitz if
$\mathcal{M}$ is metric, and $r^{p-1}$-Lipschitz if $\mathcal{M}$ is
$r$-separated, respectively. It follows either from
\eqref{cond3:ellPComplemented} or \eqref{cond5:ellPComplemented} that $
\operatorname{\mathrm{{supp}}}_{0}f_{\gamma }$ are pairwise disjoint
sets and that $f_{\gamma }(0)=0$ for every $\gamma \in \Gamma $.
Finally, we have $f_{\gamma }(x_{\gamma }) = r_{\gamma }\geq t^{-1} d(x
_{\gamma },0)$. Thus, an application of the ``Moreover'' part of
Lemma~\ref{lemma:condition2} finishes the proof.
\end{proof}

If $p<1$ and conditions~\eqref{cond4:ellPComplemented} and \eqref{cond5:ellPComplemented} hold, then
$d(0,x_{\gamma })\le t^{1/(1-p)}$ for every $\gamma \in \Gamma $. So,
if $\mathcal{M}$ is not a metric space, Lemma~\ref{lem:condition3} only
applies when $\mathcal{N}$ is bounded and $\mathcal{M}$ is uniformly
separated. Conversely, if $\mathcal{N}$ is uniformly separated and
bounded and $\mathcal{M}$ is metric (respectively, uniformly separated)
then, if $0$ is an arbitrary point of $\mathcal{N}$ and we set
\begin{equation*}
s=\inf \{ d(x,y) \colon x,y\in \mathcal{N},\, x\not =y\}, \quad R=
\sup \{ d(0,x) \colon x\in \mathcal{N}\},
\end{equation*}
the conditions of Lemma~\ref{lem:condition3} hold with $\Gamma =
\mathcal{N}\setminus \{0\}$, $r_{x}=s/2$ (resp. $r_{x}=s^{p}/2$) for
every $x\in \mathcal{N}\setminus \{0\}$ and $t=2R{s}^{-1}$ (respectively,
$t=2R{s}^{-p}$). This observation immediately yields the following
result.

\begin{Lemma}
\label{lem:bigSeparatedBoundedSubset}Suppose that $(\mathcal{M},d)$ is either a metric space or a uniformly
separated $p$-metric space, $p\in (0,1]$. If $\mathcal{N}\subset \mathcal{M}$ is
uniformly separated and bounded then:
\begin{enumerate}[(iii)]\item[(i)]
$\mathcal{F}_{p}(\mathcal{N})\simeq \ell _{p}(|\mathcal{N}|-1)$;
\item[(ii)] $L_{\jmath }\colon \mathcal{F}_{p}(\mathcal{N})\to
\mathcal{F}_{p}(\mathcal{M})$ is an isomorphic embedding; and
\item[(iii)]$L_{\jmath }(\mathcal{F}_{p}(\mathcal{N}))$ is complemented
in $\mathcal{F}_{p}(\mathcal{M})$.
\end{enumerate}
\end{Lemma}

We are almost ready to prove the second main result of this section.
Prior to do it, we write down an easy lemma which will be used several
times throughout the paper.
\begin{Lemma}
\label{lem:ToninIdea}Let $(\mathcal{M},d)$ be a $p$-metric space, $p\in (0,1]$. Suppose that  either $\mathcal M$ is unbounded or
its completion contains a limit point. Then for every $t>1$ there are
$(x_{n})_{n=1}^{\infty }$ in $\mathcal{M}$, $x_{0}$ in the completion
of $\mathcal{M}$, and a monotone sequence $(r_{n})_{n=1}^{\infty }$ in
$(0,\infty )$ such that
\begin{equation*}
d^{p}(x_{n},x_{m})\ge r_{n}+r_{m}\quad\text{and}\quad\; \frac{|r_{n} - r_{m}|}{d^{p}(x
_{n},x_{m})}\geq \frac{1}{t},
\end{equation*}
for all $m$, $n\in \mathbb{N}\cup \{0\}$, $m\not =n$,  with the convention that $r_{0}=0$.
\end{Lemma}
\begin{proof}
Choose $0<s<1$ such that $ \sqrt{t}=(1+s)/(1-s)$. In the case when
$\mathcal{M}$ is unbounded, if $x_{0}$ is an arbitrary point of
$\mathcal{M}$ there is $(x_{n})_{n=1}^{\infty }$ such that
\begin{equation*}
\sup _{n\in \mathbb{N}} \frac{d(x_{n},x_{0})}{d(x_{n+1},x_{0})} <s^{1/p}.
\end{equation*}
In the case when there is a limit point $x_{0}$ in the completion of
$\mathcal{M}$, we pick a sequence $(x_{n})_{n=1}^{\infty }$ in
$\mathcal{M}$ such that
\begin{equation*}
\sup _{n\in \mathbb{N}} \frac{d(x_{n+1},x_{0})}{d(x_{n},x_{0})} <s^{1/p}.
\end{equation*}
In both cases, if we set $r_{n}=t^{-1/2} d^{p} (x_{n},x_{0})$, we have
\begin{equation*}
r_{n}+r_{m} \le \sqrt{t} |r_{n}-r_{m}|, \quad m,n\in \mathbb{N}
\cup \{0\}, \, m\not =n.
\end{equation*}
Therefore, if $m$, $n\in \mathbb{N}\cup \{0\}$ are such that
$n\not =m$,
\begin{equation*}
d^{p}(x_{n},x_{m})
\ge | d^{p}(x_{n},x_{0})-d^{p}(x_{m},x_{0}) | =
\sqrt{t} |r_{n} -r_{m}| \ge r_{n}+r_{m},
\end{equation*}
and
\begin{equation*}
d^{p}(x_{n},x_{m})
\le d^{p}(x_{n},x_{0})+d^{p}(x_{m},x_{0}) =
\sqrt{t} (r_{n} +r_{m}) \le t |r_{n}-r_{m}|.\qedhere
\end{equation*}
\end{proof}

\begin{proof}[Proof of Theorem~\ref{thm:mainComplementedEllp2}]
Assume that $\operatorname{dens} \mathcal{M}$ is uncountable. For each $n\in \mathbb{N}$, let
$\mathcal{M}_{n}$ be a maximal $(1/n)$-separated set in $B(0,n)$. Then
$|\mathcal{M}_{n}|\le \operatorname{dens}
\mathcal{M}$ for every $n\in \mathbb{N}$, and $\mathcal{N}:=\cup _{n=1}
^{\infty }\mathcal{M}_{n}$ is dense in $\mathcal{M}$. Therefore, there
exists $n\in \mathbb{N}$ with $\mathcal{M}_{n}$ infinite. An application
of Lemma~\ref{lem:bigSeparatedBoundedSubset} yields the theorem under the assumption (b). Now, in order
to prove the result under the assumption (a), by Lemma~\ref{lem:bigSeparatedBoundedSubset}, we may assume
that $\mathcal{M}$ does not contain an infinite uniformly separated
bounded subset. We infer that either $\mathcal{M}$ is unbounded or the
completion of $\mathcal{M}$ has a limit point. Then, given $t>1$, we use
Lemma~\ref{lem:ToninIdea} to pick $(x_{n})_{n=1}^{\infty }$ in
$\mathcal{M}$, $(r_{n})_{n=1}^{\infty }$ in $(0,\infty )$ and
$x_{0}$ in the completion of $\mathcal{M}$ such that, if we choose
$x_{0}$ as the base point of $\mathcal{M}\cup \{x_{0}\}$,
\eqref{cond4:ellPComplemented} and \eqref{cond3:ellPComplemented} hold
for $\Gamma =\mathbb{N}$. This way, if $x_{0}\in \mathcal{M}$ the result
follows from Lemma~\ref{lem:condition3}. If $x_{0}\notin \mathcal{M}$
the result follows from Lemma~\ref{lem:condition3} in combination with
\cite{AACD2018}*{Proposition 4.17}.
\end{proof}

Let us note that we do not know whether
Theorem~\ref{thm:mainComplementedEllp2} holds for separable uniformly
separated $p$-metric spaces with $p<1$.
\begin{Question}
Suppose that $(\mathcal{M},d)$ is an infinite countable
uniformly separated $p$-metric space, $p\in (0,1)$. Does there exist $\mathcal{N}
\subset \mathcal{M}$ such that $\mathcal{F}_{p}(\mathcal{N})\simeq
\ell _{p}$ and $\mathcal{F}_{p}(\mathcal{N})$ is complemented in
$\mathcal{F}_{p}(\mathcal{M})$?
\end{Question}

Notice that in the case when $\operatorname{dens}
\mathcal{M}$ has uncountable cofinality, the proof of
Theorem~\ref{thm:mainComplementedEllp2} yields a subspace $
\mathcal{N}$ of $\mathcal{M}$ such that $\mathcal{F}_{p}(\mathcal{N})
\simeq \ell _{p}(\operatorname{dens} \mathcal{M})$ is
complemented in $\mathcal{F}_{p}(\mathcal{M})$. However, we do not know
if this result holds in general.

\section{Embeddability of $\ell _{p}$ in $\mathcal{F}_{p}(\mathcal{M})$ for $0<p\le 1$}\label{Sect3}
\noindent
The aim of this section is to exhibit a method (specifically tailored
for $p$-metric spaces) to linearly embed $\ell _{p}$ in $\mathcal{F}
_{p}(\mathcal{M})$ when $\mathcal{M}$ is quasimetric and $p\le 1$. Our
main result here is Theorem~\ref{thm:main}, which will be obtained as
a direct consequence of Theorems~\ref{thm:mainComplementedEllp} and \ref{thm:main2}.

\begin{Theorem}
\label{thm:main}
If $\mathcal{M}$ is an infinite $p$-metric space, $0<p\le 1$, then
$\mathcal{F}_{p}(\mathcal{M})$ contains a subspace isomorphic to
$\ell _{p}$.
\end{Theorem}

A quantitative estimate follows from Theorem~\ref{thm:james}, which is
the analogue of James's $\ell _{1}$ distortion theorem for $0<p\leq 1$.
We do not know whether for every nonseparable $p$-metric space
$\mathcal{M}$ we even have $\ell _{p}(
\operatorname{dens} \mathcal{M})\hookrightarrow \mathcal{F}_{p}(\mathcal{M})$. It
is worth it mentioning that Theorem~\ref{thm:main} gives us the
following application.

\begin{Corollary}
If $\mathcal{M}$ is an infinite $p$-metric space, $0<p<1$, then
$\mathcal{F}_{p}(\mathcal{M})$ is not $q$-convex for any $p<q\le 1$.
\end{Corollary}

As in the previous section, sometimes it is possible to have a more
precise information about the copy of $\ell _{p}$ in the space
$\mathcal{F}_{p}(\mathcal{M})$. This is the content of the following
result.

\begin{Theorem}
\label{thm:main2}
Let $0<p\leq 1$. Suppose $(\mathcal{M}, \rho )$ is a $p$-metric space
that is either unbounded, or whose completion contains a limit point.
Then for every $t>1$ there exists an infinite countable set
$\mathcal{N}\subset \mathcal{M}$ such that
\begin{enumerate}[(ii)]\item[(i)] the space $\mathcal{F}_{p}(\mathcal{N})$ is ${2^{1/p-1}}
{t}$-isomorphic to $\ell _{p}$, and
\item[(ii)] the canonical map $L_{\jmath }\colon \mathcal{F}_{p}(
\mathcal{N}) \to \mathcal{F}_{p}(\mathcal{M})$ is an isomorphic
embedding, where $\jmath \colon \mathcal{N}\to \mathcal{M}$ is the
inclusion. Quantitatively, $\Vert L_{\jmath }^{-1}\Vert \le {2^{1/p-1}}
{t}$.
\end{enumerate}
\end{Theorem}

Before giving a proof of Theorem~\ref{thm:main}, let us mention some
applications.

\begin{Corollary}
\label{cor:subspaceIsoToEllP}
Let $0<p\le 1$. Every infinite $p$-metric space $\mathcal{M}$ has a
subset $\mathcal{N}$ with $\mathcal{F}_{p}(\mathcal{N})\simeq \ell
_{p}$.
\end{Corollary}

\begin{proof}
By Theorem~\ref{thm:main2} it suffices to consider the case when
$\mathcal{M}$ is bounded and its completion is not compact. Then
$\mathcal{M}$ is not totally bounded, i.e., $\mathcal{M}$ contains a
uniformly separated infinite set, but then we are done by
\cite{AACD2018}*{Theorem 4.14}.
\end{proof}

It has been shown in \cites{CuthDoucha2016,DKP2016} that, despite the
fact that the space $\mathcal{F}(\mathcal{M})$ is isomorphic to
$\ell _{1}$ for any separable ultrametric space $\mathcal{M}$,
$\mathcal{F}(\mathcal{M})$ is never isometric to $\ell _{1}$. In
contrast, we have the following result.

\begin{Proposition}
\label{thm:ultraBM}
For each $\epsilon >0$ there is an ultrametric space $\mathcal{M}$ such
that the Banach-Mazur distance between $\mathcal{F}(\mathcal{M})$ and
$\ell _{1}$ is smaller than $1+\epsilon $.
\end{Proposition}
\begin{proof}
By Theorem~\ref{thm:main2}, in each ultrametric space that is unbounded
or contains a limit point we may find a subset $\mathcal{M}$ (which is
also ultrametric) so that the Banach-Mazur distance between
$\mathcal{F}(\mathcal{M})$ and $\ell _{1}$ is arbitrarily close to
$1$.
\end{proof}

In what remains of the section we will prove Theorems~\ref{thm:main} and \ref{thm:main2} and, for the sake of completeness, also the
above-mentioned analogue of James's $\ell _{1}$ distortion theorem for
$0<p\leq 1$. The arguments in our proofs are inspired by
\cite{CuthJohanis2017}.

\begin{Lemma}
\label{l:sum-lip}
Suppose that $X$ is a $p$-normed space and that $Y$ is a $q$-normed
space, $0<q,p\leq 1$. Let $\Gamma $ a be set, and let $f_{\gamma }\colon
X\to Y$ be an $L$-Lipschitz map for each $\gamma \in \Gamma $. Suppose
that the sets in the family $\{\operatorname{\mathrm{{supp}}}_{0}(f
_{\gamma })\}_{\gamma \in \Gamma }$ are disjoint. Then $f =
\sum _{\gamma \in \Gamma } f_{\gamma }$ is $L2^{1/q-1}$-Lipschitz.
\end{Lemma}
\begin{proof}
For each $x\in X$ there is $\gamma (x)\in \Gamma $ such that
$f_{\gamma }(x)=0$ if $\gamma \neq\gamma (x)$. In particular, the sum
defining $f(x)$ is pointwise finite and so $f$ is well-defined. Pick
$x,y\in X$ and set $\alpha =\gamma (x)$ and $\beta =\gamma (y)$. Since
the line segment $[x,y]$ is connected and the sets $
\operatorname{\mathrm{{supp}}}_{0}(f_{\alpha })$ and $
\operatorname{\mathrm{{supp}}}_{0}(f_{\beta })$ are open and disjoint,
there is $z\in [x,y]\setminus (\operatorname{\mathrm{{supp}}}_{0}(f
_{\alpha })\cup \operatorname{\mathrm{{supp}}}_{0}(f_{\beta }))$. Using
the elementary fact that
\begin{equation*}
\|y+z\|_{Y}\leq 2^{1/q-1}(\|y\|_{Y} + \|z\|_{Y}),\quad y,z\in Y,
\end{equation*}
we obtain
\begin{equation*}
\begin{split}
\|f(x)-f(y)\|_{Y}
&=\|f_{\alpha }(x)-f_{\beta }(y)\|_{Y}
\\
&=\|f_{\alpha }(x)-f_{\alpha }(z)+f_{\beta }(z)-f_{\beta }(y)\|_{Y}
\\
&\leq 2^{1/q-1}(\|f_{\alpha }(x)-f_{\alpha }(z)\|+\|f_{\beta }(z)-f
_{\beta }(y)\|_{Y})
\\
&\leq L2^{1/q-1}(\|x-z\|_{X}+\|z-y\|_{X})
\\
&=2^{1/q-1}L\|x-y\|_{X},
\end{split}
\end{equation*}
from where the conclusion follows.
\end{proof}

\begin{Lemma}
\label{lem:33}Let $(\mathcal{M},d)$ be a pointed $p$-metric space, $0<p\le 1$. Assume that
$(x_{n})_{n=0}^{\infty }$ in $\mathcal{M}$ and $(r_{n})_{n=0}^{\infty
}$ in $[0,+\infty )$ are such that
\begin{equation}
\label{eq:27}
d^{p}(x_{m},x_{n})\geq r_{m}+r_{n}, \quad m,n\in \mathbb{N}\cup \{0\},
\, m\neq n.
\end{equation}
Then there is a sequence $(f_{n})_{n=0}^{\infty }$ of $1$-Lipschitz maps
from $(\mathcal{M},d)$ into $L_{p}(\mathbb{R})$ such that $f_{n}(x
_{n})=\chi _{(0,r_{n}]}$ for every $n\in \mathbb{N}\cup \{0\}$ and
$(\operatorname{\mathrm{{supp}}}_{0}(f_{n}))_{n=0}^{\infty }$ is a
pairwise disjoint sequence.
\end{Lemma}
\begin{proof}
For $n\in \mathbb{N}\cup \{0\}$ define $g_{n}\colon \mathcal{M}\to
\mathbb{R}$ by
\begin{equation}
\label{eq:28}
g_{n}(x)=\max \{r_{n}-d^{p}(x,x_{n}),0\}, \quad x\in \mathcal{M}.
\end{equation}
It is straightforward to check that the sets $\{
\operatorname{\mathrm{{supp}}}_{0}(g_{n})\}_{n=0}^{\infty }$ are
pairwise disjoint, that $g(x_{n})=r_{n}$, and that $g_{n}\colon (
\mathcal{M},\rho )\to (\mathbb{R},|\cdot |^{1/p}) $ is $1$-Lipschitz.
Then, if we consider the standard embedding
\begin{equation*}
\Phi \colon (\mathbb{R},|\cdot |^{1/p})\to L_{p}(\mathbb{R}), \quad
\Phi (x)=\chi _{(0,x]},
\end{equation*}
the sequence $(\Phi \circ g_{n})_{n=0}^{\infty }$ has the desired
properties.
\end{proof}

The next proposition is the last ingredient that we need for proving the
main results of this section.

\begin{Proposition}
\label{Thm:AnsoMichaLast}
Let $0<p\leq 1$. Suppose $(\mathcal{M}, d)$ is a $p$-metric space
containing a sequence $(x_{n})_{n=0}^{\infty }$ such that \eqref{eq:27} holds for some
monotone sequence $(r_{n})_{n=0}^{\infty }$ in $[0,+\infty )$ satisfying
\begin{equation*}
\frac{|r_{n} - r_{n-1}|}{d^{p}(x_{n},x_{n-1})}\geq \frac{1}{t},\quad
n\in \mathbb{N},
\end{equation*}
for some $t>0$. Then, if we set $\mathcal{N}=\{x_{n} \colon n\in
\mathbb{N}\cup \{0\}\}$, we have:
\begin{enumerate}[(ii)]\item[(i)] The space $\mathcal{F}_{p}(\mathcal{N})$ is
${2^{1/p-1}}t$-isomorphic to $\ell _{p}$, and
\item[(ii)] The canonical map $L_{\jmath }\colon \mathcal{F}_{p}(
\mathcal{N}) \to \mathcal{F}_{p}(\mathcal{M})$ is an isomorphic
embedding, where $\jmath \colon \mathcal{N}\to \mathcal{M}$ is the
inclusion. Quantitatively,
$\Vert L_{\jmath }^{-1}\Vert \le {2^{1/p-1}}t$.
\end{enumerate}
\end{Proposition}
\begin{proof}
We may assume that $x_{0}$ is the base point of $\mathcal{M}$. For each
$n\in \mathbb{N}$ put
\begin{align*}
v_{n}
&= \frac{\delta _{\mathcal{N}}(x_{n-1}) - \delta _{\mathcal{N}}(x
_{n})}{d(x_{n-1},x_{n})}\in \mathcal{F}_{p}(\mathcal{N}),
\\
\quad b_{n}
&= \frac{\delta _{\mathcal{M}}(x_{n-1}) - \delta _{
\mathcal{M}}(x_{n})}{d(x_{n-1},x_{n})}\in \mathcal{F}_{p}(\mathcal{M}).
\end{align*}
There is a norm-one linear map $T\colon \ell _{p}\to \mathcal{F}_{p}(
\mathcal{N})$ such that $T(\mathbf{e}_{n})=v_{n}$ for all $n\in
\mathbb{N}$. Since $L_{\jmath }(v_{n})=b_{n}$ for $n\in \mathbb{N}$ and
$\overline{\operatorname{\mathrm{{span}}}}\{v_{n}\colon n\in
\mathbb{N}\}=\mathcal{F}_{p}(\mathcal{N})$, it suffices to prove that
$(b_{n})_{n=1}^{\infty }$ is ${2^{1/p-1}}{t}$-dominated by the canonical
basis of $\ell _{p}$.

We consider $\mathcal{M}$ as a subset of $X=\mathcal{F}_{p}(
\mathcal{M})$. Let $(f_{n})_{n=0}^{\infty }$ be the sequence of maps
from $X$ into $L_{p}(\mathbb{R})$ provided by Lemma~\ref{lem:33}. Pick
any $a_{1},\dotsc ,a_{N}\in \mathbb{R}$ and consider $x=\sum _{n=1}
^{N} a_{n}\, b_{n}$. Let $g\colon X \to L_{p}(\mathbb{R})$ be defined
for $z\in \mathcal{F}_{p}(\mathcal{M})$ by
\begin{equation*}
g(z)=\sum _{n=0}^{N} f_{n}(z).
\end{equation*}
Lemma~\ref{l:sum-lip} yields that $g \colon X \to L_{p}(\mathbb{R})$ is
$2^{1/p-1}$-Lipschitz. Hence, $f = g|_{\mathcal{M}}$ is
$2^{1/p-1}$-Lipschitz. Since $f(x_{n})=\chi _{(0,r_{n}]}$ we have
\begin{align*}
\|x\|
& \geq \frac{2}{2^{1/p}}\left \|  \sum _{n=1}^{N} a_{n} \frac{ f(x
_{n-1})-f(x_{n})}{d(x_{n-1},x_{n})}\right \|  _{L_{p}}
\\
&=\frac{2}{2^{1/p}}\left \|  \sum _{n=1}^{N} a_{n} \frac{
\chi _{(r_{n-1},r_{n}]}}{d(x_{n-1},x_{n})}\right \|  _{L_{p}}
\\
& = \frac{2}{2^{1/p}} \left (\sum _{n=1}^{N}|a_{n}|^{p}\frac{|r_{n-1}
- r_{n}|}{d^{p}(x_{n-1},x_{n})}\right )^{1/p}
\\*
&\geq \frac{2}{2^{1/p}} \frac{1}{t} \left (\sum _{n=1}^{N}|a_{n}|^{p}\right )
^{1/p}.\qedhere
\end{align*}
\end{proof}

\begin{proof}[Proof of Theorem~\ref{thm:main2}]
It follows by combining 
Proposition~\ref{Thm:AnsoMichaLast} with Lemma~\ref{lem:ToninIdea}.
\end{proof}

\begin{proof}[Proof of Theorem~\ref{thm:main}]
By Theorem~\ref{thm:discretePoints} we can assume that there is a limit
point in the completion of $\mathcal{M}$. Then, the result follows from
Theorem~\ref{thm:main2}.
\end{proof}

As we advertised, there is a non-locally convex version of James's
$\ell _{1}$ distortion theorem. As far as we know, this is the first time
that the validity of this result is explicitly stated and so we include
its proof for further reference.

\begin{Theorem}[James's $\ell _{p}$ distortion theorem for $0<p\le 1$]
\label{thm:james}
Let $X$ be a $p$-Banach space containing a normalized basic sequence  $(x_{j})_{j=1}^{\infty }$ 
which is equivalent to the canonical
$\ell _{p}$-basis, $0<p\le 1$. Then given $\epsilon >0$ there is a
normalized block basic sequence $(y_{k})_{k=1}^{\infty }$ of
$(x_{j})_{j=1}^{\infty }$ such that
\begin{equation*}
\left \|  \sum _{k=1}^{\infty }a_{k} y_{k}\right \|  \ge (1-\epsilon )
\left (\sum _{k=1}^{\infty }|a_{k}|^{p}\right )^{1/p},
\end{equation*}
for any sequence of scalars $(a_{k})_{k=1}^{\infty }\in c_{00}$.
\end{Theorem}

\begin{proof}
By hypothesis, there is $M_{0}<\infty $ such that
\begin{equation*}
\left (\sum _{j=1}^{\infty }|a_{j}|^{p}\right )^{1/p}\le M_{0}\left \|
\sum _{j=1}^{\infty }a_{j} x_{j}\right \|
\end{equation*}
for every $(a_{j})_{j=1}^{\infty }\in c_{00}$. Hence, for each integer
$n$ we can consider the least constant $M_{n}$ so that if $(a_{j})_{j=1}
^{\infty }\in c_{00}$ with $a_{j}=0$ for $j\le n$ then
\begin{equation*}
\left (\sum _{j=1}^{\infty }|a_{j}|^{p}\right )^{1/p}\le M_{n}\left \|
\sum _{j=1}^{\infty }a_{j} x_{j}\right \|  .
\end{equation*}
The sequence $(M_{n})_{n=1}^{\infty }$ is decreasing, and the inequality
\begin{equation*}
\left \Vert \sum _{j=1}^{\infty }a_{j}x_{j}\right \Vert \le \left (\sum
_{j=1}^{\infty }|a_{j}|^{p}\right )^{1/p}
\end{equation*}
yields $M_{n}\ge 1$. Let $M=\lim _{n\to \infty }M_{n}\ge 1$. We
recursively construct an increasing sequence $(n_{k})_{k=0}^{\infty }$
of positive integers and a sequence $(y_{k})_{k=1}^{\infty }$ in $X$.
We start by choosing $n_{0}\in \mathbb{N}$ such that
\begin{equation*}
M_{n_{0}}\le (1-\epsilon )^{-1/2}M.
\end{equation*}
Let $k\in \mathbb{N}$ and assume that $n_{i}$ and $y_{i}$ are
constructed for $i<k$. Since $(1-\epsilon )^{1/2}M <M_{n_{k-1}}$ there
is $n_{k}>n_{k-1}$ and a norm-one vector $ y_{k}=\sum _{j=1+n_{k-1}}
^{n_{k}}b_{j}\, x_{j}\in X $ such that
\begin{equation*}
\left ( \sum _{j=1+n_{k-1}}^{n_{k}}|b_{j}|^{p}\right )^{1/p}\ge (1-
\epsilon )^{1/2}M.
\end{equation*}
The normalized block basic sequence $(y_{k})_{k=1}^{\infty }$ satisfies
the desired property. Indeed, for any $(a_{k})_{k=1}^{\infty }\in c
_{00}$ we have
\begin{align*}
\left \|  \sum _{k=1}^{\infty }a_{k}\, y_{k}\right \|
& =\left \|
\sum _{k=1}^{\infty }\sum _{j=1+n_{k-1}}^{n_{k}} a_{k} \, b_{j} \, x
_{j}\right \|
\\
&\ge M_{n_{0}}^{-1} \left (\sum _{k=1}^{\infty }\sum _{j=1+n_{k-1}}
^{n_{k}} |a_{k}|^{p} \, |b_{j}|^{p}\right )^{1/p}
\\
&\ge (1-\epsilon )^{1/2}M M_{n_{0}}^{-1} \left (\sum _{k=1}^{\infty }|a
_{k}|^{p}\right )^{1/p}
\\
&\ge (1-\epsilon ) \left (\sum _{k=1}^{\infty }|a_{k}|^{p}\right )
^{1/p}. \qedhere
\end{align*}
\end{proof}

In \cite{DRT} the following improvement of James's $\ell _{1}$
distorsion theorem is obtained: whenever a Banach space $X$ has a
complemented subspace isomorphic to $\ell _{1}$, then $X$ has for each
$\varepsilon >0$ a $(1+\varepsilon )$-complemented subspace
$(1+\varepsilon )$--isomorphic to $\ell _{1}$. Since the proof of this
fact relies on duality techniques we wonder whether there is an analogue
for $p$-Banach spaces.

\begin{Question}
Does there exist $C\geq 1$ such that whenever a $p$-Banach space $X$ has
a complemented subspace isomorphic to $\ell _{p}$, then it has a
$C$-complemented subspace $C$--isomorphic to $\ell _{p}$?
\end{Question}

\section{Bases in $\mathcal{F}_{p}(\mathbb{N})$ and $\mathcal{F}_{p}([0,1])$}\label{Sect:bases}
\noindent
Given $0<p\le 1$, $n\in \mathbb{N}$, and $K\subseteq \mathbb{R}^{n}$ let
us denote by $\mathcal{F}_{p}(K)$ the Lipschitz free $p$-space over
$K$ equipped with the Euclidean metric. In this section we address the
study of Schauder bases for the Lipschitz free $p$-spaces $
\mathcal{F}_{p}(\mathbb{N})$ and $\mathcal{F}_{p}([0,1])$. If $p=1$
those spaces are isometric to $\ell _{1}$ and $L_{1}$, respectively.
However, for $p<1$ we obtain an interesting family of non-classical
$p$-Banach spaces. One of the most important results of this section is
that $\mathcal{F}_{p}([0,1])$ has a Schauder basis for $0<p<1$.
This provides examples of $p$-Banach spaces that are not Banach spaces, which have a basis but cannot have an unconditional basis.
Indeed, by \cite{AACD2018}*{Proposition 4.20}, the Banach envelope of
$\mathcal{F}_{p}([0,1])$ is isometrically isomorphic to $\mathcal{F}(
[0,1] )\simeq L_{1}[0,1]$, which does not have an unconditional basis
(see, e.g., \cite{AlbiacKalton2016}*{Theorem 6.3.3}).
Therefore, $\mathcal{F}_{p}([0,1])$ cannot have an unconditional basis either
(see \cite{AABW2019}*{Proposition 9.9}).

First, we deal with $\mathcal{F}_{p}( \mathbb{N}_{*})$, where
$\mathbb{N}_{*}:=\mathbb{N}\cup \{0\}$. Of course, $\mathbb{N}_{*}$ and
$\mathbb{N}$ are isometric as metric spaces, but we prefer to work with
the former and choose $0$ as its base point.

\begin{Lemma}
\label{lem:bases:1}Let $(a_{n})_{n=1}^{\infty }$ be an eventually null sequence of scalars.
Then
\begin{equation*}
\left \Vert \sum _{n=1}^{\infty }a_{n}\, \delta (n) \right \Vert _{
\mathcal{F}_{p}(\mathbb{N}_{*})} \ge \left (\sum _{\substack{n\in
\mathbb{N}
\\
a_{n}\ge 0}} a_{n}^{p}\right )^{1/p}.
\end{equation*}
\end{Lemma}
\begin{proof}
Let $\alpha $ be the $\{0,1\}$-metric on $\mathbb{N}_{*}$. Since
$\alpha (n,m)\le | n-m|$ for every $n,m\in \mathbb{N}_{*}$, we have
\begin{equation*}
\Vert f \Vert _{\mathcal{F}_{p}(\mathbb{N}_{*},\alpha )}\le \Vert f
\Vert _{\mathcal{F}_{p}(\mathbb{N}_{*})}
\end{equation*}
for every $f\in \mathcal{P}(\mathbb{N}_{*})$. An appeal to
\cite{AACD2018}*{Proposition 4.16} completes the proof.
\end{proof}
By $\mathbb{Z}[k,m]$ we denote the set $\{n\in \mathbb{Z}\colon k
\le n \le m\}$ whenever $k,m\in \mathbb{Z}$. Let us see a preliminary
result that we will need.
Notice that a Schauder basis $\mathcal{B}$ of a quasi-Banach space $X$ is also a basis when regarded inside the Banach envelope of $X$
(see \cite{AABW2019}*{Proposition 9.9}).

\begin{Theorem}
\label{thm:BasisFpN}Let $0<p\le 1$ and $\mathcal{B}=(\mathbf{x}_{n})_{n=1}^{\infty }$ be the
sequence in $\mathcal{F}_{p}(\mathbb{N}_{*})$ defined by $\mathbf{x}
_{n}=\delta (n)-\delta (n-1)$. Then:
\begin{enumerate}[(a)]\item[(a)]$\mathcal{B}$ is a normalized bi-monotone basis of
$\mathcal{F}_{p}(\mathbb{N}_{*})$.
\item[(b)] For every $m\in \mathbb{N}$, $
\operatorname{\mathrm{{span}}}(\{ \mathbf{x}_{n} \colon 1\le n \le m
\})\simeq \mathcal{F}_{p}(\mathbb{Z}[0,m])$ isometrically.
\item[(c)] The subbases $(\mathbf{x}_{2k-1})_{n=1}^{\infty }$ and
$(\mathbf{x}_{2k})_{n=1}^{\infty }$ are isometrically equivalent to the
unit vector basis of $\ell _{p}$.
\item[(d)]If $1\le k \le m$, then $\Vert \sum _{n=k+1}^{m} \mathbf{x}
_{n}\Vert = m-k$.
\item[(e)] If $p<1$ then $\mathcal{B}$ is a conditional basis.
\item[(f)]The sequence $\mathcal{B}$, regarded as a basis of the Banach
envelope of $\mathcal{F}_{p}(\mathbb{N}_{*})$, is isometrically
equivalent to the unit vector basis of $\ell _{1}$.
\end{enumerate}
\end{Theorem}
\begin{proof}
Let $0\le k <m$. The map
\begin{equation*}
r[k,m]\colon \mathbb{N}_{*} \to \mathbb{N}_{*} \colon n \mapsto
\max \{ k, \min \{n,m\}\}
\end{equation*}
is a $1$-Lipschitz retraction from $\mathbb{N}_{*}$ onto
$\mathbb{Z}[k,m]$. Then there is a norm-one linear projection
$P[k,m]$ from $\mathcal{F}_{p}(\mathbb{N}_{*})$ onto itself such that
\begin{equation*}
P[k,m](\delta (n))=
\begin{cases}
0
& \text{ if } n \le k,
\\
\delta (m)-\delta (k)
& \text{ if } n\ge m,
\\
\delta (n)-\delta (k)
& \text{ if } k\le n\le m,
\end{cases}
\end{equation*}
and an isometric linear embedding $L[k,m]\colon \mathcal{F}_{p}(
\mathbb{Z}[k,m])\to \mathcal{F}_{p}(\mathbb{N}_{*})$ such that, if we
choose $k$ as the base point of $\mathbb{Z}[k,m]$,
\begin{equation*}
L[k,m](\delta (n))=\delta (n)-\delta (k).
\end{equation*}
We deduce from the case $k=0$ that (b) holds. Moreover, if we put
$P_{m}=P[0,m]$ and $V_{m}=P_{m}(\mathcal{F}_{p}(\mathbb{N}_{*}))$, we
have
\begin{equation*}
V_{m}=\operatorname{\mathrm{{span}}}(\{ \delta (n)\colon 1\le n
\le m \})
=\operatorname{\mathrm{{span}}}(\{ \mathbf{x}_{n} \colon 1
\le n \le m \}).
\end{equation*}
Consequently, $\dim V_{m}=m$ and $\cup _{m=1}^{\infty }V_{m}$ is dense
is $\mathcal{F}_{p}(\mathbb{N}_{*})$. It is clear that $P_{m}\circ P
_{j}=P_{\min \{ j, m \}}$ and that $P_{m}(\mathbf{x}_{n})=\mathbf{x}
_{n}$ if $m>n$. Hence, $\mathcal{B}$ is a Schauder basis with
partial-sum projections $(P_{m})_{m=1}^{\infty }$. As
$P_{m}-P_{k}=P[k+1,m]$, the basis $\mathcal{B}$ is bi-monotone. Since
$\sum _{n=k+1}^{m} \mathbf{x}_{n}=\delta (m)-\delta (k)$, we have $
\left \Vert \sum _{n=k+1}^{m} \mathbf{x}_{n}\right \Vert = m-k $ for
every $0\le k<m$. In particular, $\Vert {\mathbf{x}_{n}}
\Vert _{\mathcal{F}_{p}(\mathbb{N}_{*})}=1$ for every $n\in \mathbb{N}$.
Summarizing, we have proved that (a) and (d) hold.

Let $(a_{k})_{k=1}^{\infty }$ be eventually zero and consider
$A=\{ k \colon a_{k} > 0\}$ and $B=\{ k \colon a_{k} < 0\}$. 
Lemma~\ref{lem:bases:1} and $p$-convexity yield
\begin{align*}
\left ( \sum _{k=1}^{\infty }|a_{k}|^{p} \right )^{1/p}
&\ge \left \Vert \sum
_{k=1}^{\infty }a_{k} \, \mathbf{x}_{2k}\right \Vert _{\mathcal{F}
_{p}(\mathbb{N}_{*})}
\\
&=\left \Vert \sum _{k=1}^{\infty }a_{k} \, \delta (2k) - \sum _{k=1}
^{\infty }a_{k} \, \delta (2k-1) \right \Vert _{\mathcal{F}_{p}(
\mathbb{N}_{*})}
\\
&\ge \left ( \sum _{k\in A} a_{k}^{\, p} + \sum _{k\in B} (-a_{k})^{p}
\right )^{1/p}
\\
&= \left ( \sum _{k=1}^{\infty }|a_{k}|^{p} \right )^{1/p}.
\end{align*}
For the odd terms, we proceed analogously. Since $\mathcal{B}$ is
normalized, (c) holds. Part (e) is a consequence of the previous ones.
Indeed, if $\mathcal{B}$ were unconditional we would have
\begin{align*}
\left \Vert \sum _{k=1}^{\infty }a_{n} \, \mathbf{x}_{n}\right \Vert _{\mathcal{F}_{p}(\mathbb{N}_{*})}^{p}
&\approx \left \Vert \sum _{k=1}
^{\infty }a_{2k-1} \, \mathbf{x}_{2k-1}\right \Vert _{\mathcal{F}_{p}(
\mathbb{N}_{*})}^{p}
+\left \Vert \sum _{k=1}^{\infty }a_{2k} \,
\mathbf{x}_{2k}\right \Vert _{\mathcal{F}_{p}(\mathbb{N}_{*})}^{p}
\\
&=\sum _{n=1}^{\infty }|a_{n}|^{p}
\end{align*}
for all sequences of scalars $(a_{n})_{n=1}^{\infty }$ eventually zero.
In particular, by (d) we would have $m \approx m^{p}$ for $m\in
\mathbb{N}$, which is false unless $p=1$.

By \cite{AACD2018}*{Proposition 4.20}, $\mathcal{B}$ is, when regarded
in the Banach envelope, the sequence $(\delta (n)-\delta (n-1))_{n=1}
^{\infty }$ of the Banach space $\mathcal{F}(\mathbb{N}_{*})$. Hence,
(f) is known and follows, e.g., from the more general
\cite{Godard2010}*{Proposition 2.3}.
\end{proof}

As the attentive reader might have noticed, a result similar to
Theorem~\ref{thm:BasisFpN} holds for $\mathcal{F}_{p}(\mathbb{Z})$. Let
us point out that the existence of a Schauder basis for $\mathcal{F}
_{p}(\mathbb{Z})$ can also be deduced from the following general result.

\begin{Theorem}
\label{rem:freeZetD}
Let $\mathcal{M}$ be either a net on $c_{0}$ or a net on $\mathbb{R}
^{n}$, $n\in \mathbb{N}$. Then $\mathcal{F}_{p}(\mathcal{M})$ has a
Schauder basis.
\end{Theorem}
\begin{proof}
The proof can be carried out exactly as in \cite{HN17}*{Corollaries~16 and 18}, with the exceptions that instead of
\cite{HN17}*{Proposition 5} we use Proposition~\ref{prop:nets} and
that we need to prove \cite{HN17}*{Theorem~13} also for $p<1$. As a matter of fact, the proof of \cite{HN17}*{Theorem~13} uses the
universal property of Lipschitz free spaces only, so the same proof
works even for $p<1$.
\end{proof}

Now, let us mention a preliminary result which concerns the structure
of $\mathcal{F}_{p}([0,1])$.

\begin{Lemma}
\label{lem:keyinterval}For each pair $(K_{1},K_{2})$ with $\{0,1\}\subseteq K_{2}\subseteq K
_{1}\subseteq [0,1]$ there is a linear operator $P_{K_{1},K_{2}}
\colon \mathcal{F}_{p}(K_{1})\to \mathcal{F}_{p}(K_{2})$ such that, if
$L_{K_{1},K_{2}}$ denotes the canonical linear map from $\mathcal{F}
_{p}(K_{2})$ into $\mathcal{F}_{p}(K_{1})$ and $ \{0,1\} \subseteq K
_{3}\subseteq K_{2}\subseteq K_{1}\subseteq [0,1]$,
\begin{enumerate}[(iii)]\item[(i)] $\Vert P_{K_{1},K_{2}}\Vert \le 3^{1/p-1}$ and $P_{K_{1},K
_{2}}\circ L_{K_{1},K_{2}} = \operatorname{\mathrm{{Id}}}_{
\mathcal{F}_{p}(K_{2})}$;
\item[(ii)] $P_{K_{2},K_{3}}\circ P_{K_{1},K_{2}}=P_{K_{1},K_{3}}$; and
\item[(iii)] $P_{K_{1},K_{3}}\circ L_{K_{1},K_{2}}=P_{K_{2},K_{3}}$ and
$P_{K_{1},K_{2}}\circ L_{K_{1},K_{3}}=L_{K_{2},K_{3}}$.
\end{enumerate}
Moreover, if $a<x<b$ are such that $[a,b] \cap K_{2}=\{ a, b\}$ and
$x\in K_{1}$, we have
\begin{equation*}
P_{K_{1},K_{2}} (\delta _{K_{1}}(x))=\frac{b-x}{b-a}\delta _{K_{2}}(a) +
\frac{x-a}{b-a} \delta _{K_{2}}(b).
\end{equation*}
\end{Lemma}
\begin{proof}
By \cite{AACD2018}*{Proposition 4.17} we can assume that $K_{i}$ is
closed for $i\in \{1,2,3\}$. Since for $0\le a \le b \le 1$ the mapping
\begin{equation*}
x\mapsto \max \{ a, \min \{x,b\}\}
\end{equation*}
is a $1$-Lipschitz retraction from $[0,1]$ onto $[a,b]$ we can assume
that $\{0,1\}\subseteq K_{3}$. Given $x\in [0,1]\setminus K_{2}$ there
are $a=a[x,K_{2}]$, and $b=b[x,K_{2}]$ such that $\{ a,b\}\subseteq K
_{2}$ and $x\in (a,b)\cap K_{2}=\emptyset $. Define $f=f_{K_{1},K_{2}}
\colon K_{1}\to \mathcal{F}_{p}(K_{2})$ by
\begin{equation}
\label{eq:2}
f_{K_{1},K_{2}}(x)=
\begin{cases}
\frac{b-x}{b-a}\delta _{K_{2}}(a) + \frac{x-a}{b-a} \delta _{K_{2}}(b)
&
\text{ if } x\in K_{1}\setminus K_{2},
\\
\ \delta _{K_{2}}(x)
& \text{ if } x\in K_{2}.
\end{cases}
\end{equation}
Let $x$, $y\in K_{1}$ with $x<y$. If $x$, $y\in [a,b]$ with $a$,
$b\in K_{2}$ and $(a,b)\cap K_{2}=\emptyset $ we have
\begin{equation*}
\Vert f(x)-f(y)\Vert =\left \Vert \frac{x-y}{b-a} (\delta _{K_{2}}(b)-
\delta _{K_{2}}(a) ) \right \Vert =|x-y|.
\end{equation*}
In general, there are $a$, $b$, $c$, $d\in K_{2}$ such that
$a\le x \le b \le c \le y \le d$. Then
\begin{align*}
\Vert f(x)-f(y)\Vert ^{p}
&=\Vert f(x)-f(b)\Vert ^{p} + \Vert f(b)-f(c)
\Vert ^{p} + \Vert f(c)-f(y)\Vert ^{p}
\\
&\le |x-b|^{p}+|b-c|^{p}+|c-y|^{p}
\\
&\le 3^{1-p}|y-x|^{p}.
\end{align*}
Hence, by \cite{AACD2018}*{Theorem 4.5}, there is $P_{K_{1},K_{2}}
\colon \mathcal{F}_{p}(K_{1})\to \mathcal{F}_{p}(K_{2})$ such that
$\Vert P_{K_{1},K_{2}}\Vert \le 3^{1/p-1}$ and
\begin{equation}
\label{eq:1}
P_{K_{1},K_{2}}\circ \delta _{K_{1}} =f_{K_{1},K_{2}}.
\end{equation}
If $x\in K_{2}$ we have
\begin{equation*}
P_{K_{1},K_{2}} (L_{K_{1},K_{2}} (\delta _{K_{2}}(x)))=f_{K_{1},K_{2}}(x)=
\delta _{K_{2}}(x),
\end{equation*}
and so (i) holds. In order to prove (ii), we pick $x\in K_{1}$. In the
case when $x\in K_{2}$ it is clear from \eqref{eq:2} and \eqref{eq:1}
that
\begin{equation*}
P_{K_{2},K_{3}} (P_{K_{1},K_{2}} (\delta _{K_{1}}(x)))=P_{K_{2},K_{3}}(
\delta _{K_{2}}(x))=f_{K_{1},K_{3}}(x).
\end{equation*}
Assume that $x\notin K_{2}$ and set $a_{2}=a[x,K_{2}]$, $b_{2}=b[x,K
_{2}]$, $a_{3}=a[x,K_{3}]$ and $b_{3}=b[x,K_{3}]$. We have
\begin{align*}
P_{K_{2},K_{3}} (P_{K_{1},K_{2}} (\delta _{K_{1}}(x)))
&= \frac{b_{2}-x}{b
_{2}-a_{2}} \left ( \frac{b_{3}-a_{2}}{b_{3}-a_{3}} \delta _{K_{3}}(a
_{3}) + \frac{a_{2}-a_{3}}{b_{3}-a_{3}} \delta _{K_{3}}(b_{3}) \right )
\\
&+\frac{x-a_{2}}{b_{2}-a_{2}} \left (
\frac{b_{3}-b_{2}}{b_{3}-a_{3}} \delta _{K_{3}}(a_{3}) + \frac{b_{2}-a
_{3}}{b_{3}-a_{3}} \delta _{K_{3}}(b_{3}) \right )
\\
&= \frac{b_{3}-x}{b_{3}-a_{3}} \delta _{K_{3}}(a_{3}) + \frac{x-a_{3}}{b
_{3}-a_{3}} \delta _{K_{3}}(b_{3})
\\
&=f_{K_{1},K_{3}}(x).
\end{align*}
Thus, (ii) holds.

(iii) is a straightforward consequence of (i) and (ii).
\end{proof}

The following result is a version (and a generalization) of the fact
that conditional expectations define bounded operators in $L_{1}$.
\begin{Theorem}
\label{thm:conditionalexpectation}
For any $K\subseteq [0,1]$, the space $\mathcal{F}_{p}(K)$ is
complemented in $\mathcal{F}_{p}([0,1])$. To be precise, there is a
linear map $P\colon \mathcal{F}_{p}([0,1])\to \mathcal{F}_{p}(K)$ such
that $\Vert P\Vert \le 3^{1/p-1}$ and, if $\jmath \colon K \to [0,1]$
is the inclusion map, $P\circ L_{\jmath }=
\operatorname{\mathrm{{Id}}}_{\mathcal{F}_{p}(K)}$.
\end{Theorem}
\begin{proof}
Since, for $0\le a \le b \le 1$, the mapping
\begin{equation*}
x\mapsto \max \{ a, \min \{x,b\}\}
\end{equation*}
is a $1$-Lipschitz retraction from $[0,1]$ onto $[a,b]$ we can assume
that $\{0,1\}\subseteq K$ (see, e.g., \cite{AACD2018}*{Lemma 4.19}).
Now the result is immediate from Lemma~\ref{lem:keyinterval}.
\end{proof}

\begin{Theorem}
\label{thm:crudely}
Let $0<p\le 1$. The space $\mathcal{F}_{p}([0,1])$ is crudely finitely  representable in
$\mathcal{F}_{p}(\mathbb{N}_{*})$, and the space $\mathcal{F}_{p}(
\mathbb{N}_{*})$ is crudely finitely representable in $\mathcal{F}
_{p}([0,1])$. That is, the finite dimensional subspace structures of the $p$-Banach spaces
$\mathcal{F}_{p}([0,1])$ and $\mathcal{F}_{p}(\mathbb{N}_{*})$ coincide.
\end{Theorem}

\begin{proof}
Let $K_{n}=\{i 2^{-n} \colon 0\le i \le 2^{n}\}$, $X_{n}:=
\operatorname{\mathrm{{span}}}\left (\{ \delta _{[0,1]}(x) \colon x
\in K_{n}\right \}  )$, $\mathcal M_{n}=\mathbb{Z}[0,2^{n}]$ and $Y_{n}:=
\operatorname{\mathrm{{span}}}\left (\{ \delta _{\mathbb{N}_{*}}(x)
\colon x\in M_{n}\}\right )$. By
Theorem~\ref{thm:conditionalexpectation}, $X_{n}$ is uniformly
isomorphic to $\mathcal{F}_{p}(K_{n})$, and by
Theorem~\ref{thm:BasisFpN}, $Y_{n}$ is uniformly isomorphic to
$\mathcal{F}_{p}(\mathcal M_{n})$. Moreover $\cup _{n=1}^{\infty }X_{n}$ is dense
in $\mathcal{F}_{p}([0,1])$ and $\cup _{n=1}^{\infty }Y_{n}$ is dense in
$\mathcal{F}_{p}(\mathbb{N}_{*})$. Since $K_{n}$ is Lipschitz-isomorphic
to $\mathcal M_{n}$ with distorsion constant one, we are done.
\end{proof}

Next, we generalize the fact that the Haar system is a Schauder basis
of $L_{1}([0,1])$. Given an interval $J=[a,b]\subseteq [0,1]$ we define
the Haar molecule $h_{J}$ of the interval $J$ by
\begin{align*}
h_{J}
&=-\left (\delta _{[0,1]}\left (\frac{a+b}{2}\right )-\delta _{[0,1]}(a)\right )+
\left (\delta _{[0,1]}(b)-\delta _{[0,1]}\left (\frac{a+b}{2}\right ) \right )
\\
&=\delta _{[0,1]}(a)+\delta _{[0,1]}(b)-2\delta _{[0,1]}\left (\frac{a+b}{2}\right ).
\end{align*}
Denote also $h_{0}=\delta _{[0,1]}(1)-\delta _{[0,1]}(0)$. The Haar system
of $\mathcal{F}_{p}(\mathcal{M})$ is the family $\mathcal{H}=(h_{J})_{
\{0\}\cup \mathcal{D}}$, where $\mathcal{D}$ is the set of dyadic
intervals contained in $[0,1]$.

\begin{Theorem}
\label{thm:basisFpR}The Haar system, arranged if such a way that $h_{0}$ is its first term
and Haar molecules of bigger intervals appear before, is a Schauder
basis of $\mathcal{F}_{p}([0,1])$ with basis constant not bigger than
$3^{1/p-1}$.
\end{Theorem}

\begin{proof}
Let $(\mathbf{x}_{n})_{n=0}^{\infty }$ be such an arrangement of the
Haar system of $\mathcal{F}_{p}(\mathcal{M})$. Let $(K_{n})_{n=0}^{
\infty }$ be the sequence of subsets of $[0,1]$ constructed recursively
as follows: Put $K_{0}=\{0,1\}$, and if $\mathbf{x}_{n}=h_{J}$ and $c$ is
the middle point of $J$ then $K_{n}=K_{n-1}\cup \{ c \}$. By induction we
see that
\begin{equation*}
\mathbf{x}_{k}\in X_{n}:=\operatorname{\mathrm{{span}}}(\{\delta (x)
\colon x\in K_{n}\}), \quad 0\le k \le n.
\end{equation*}
Since the dyadic points are dense in $[0,1]$, $\cup _{n=0}^{\infty }X
_{n}$ is dense in $\mathcal{F}_{p}([0,1])$ by
\cite{AACD2018}*{Proposition~4.17}. Using the notation of
Lemma~\ref{lem:keyinterval}, we set
\begin{equation*}
P_{n}=L_{[0,1],K_{n}} \circ P_{[0,1],K_{n}}, \quad n\ge 0.
\end{equation*}
It is clear that $P_{n}(\mathcal{F}_{p}([0,1]))=X_{n}$. By
Lemma~\ref{lem:keyinterval} (i), $\Vert P_{n}\Vert \le 3^{1/p-1}$ and
by Lemma~\ref{lem:keyinterval} (ii) and (iii), $P_{n}\circ P_{m}=P
_{\min \{n,m\}}$. Moreover, if $\mathbf{x}_{n+1}=h_{J}$ and
$J=[a,b]$,
\begin{align*}
P_{n}(\mathbf{x}_{n+1})
&=P_{n}\left (\delta (a)+\delta (b)-2\delta
\left ( \frac{a+b}{2}\right )\right )
\\
&=\delta (a)+\delta (b)-2\left ( \frac{1}{2} \delta (a)+\frac{1}{2}
\delta (b)\right )=0.
\end{align*}
Thus $(\mathbf{x}_{n})_{n=0}^{\infty }$ is a Schauder basis for
$\mathcal{F}_{p}([0,1])$ with partial-sum projections $(P_{n})_{n=0}
^{\infty }$.
\end{proof}

\section{Questions and remarks}
\noindent
In this last section we gather a few questions that arise naturally from
our work, and which suggest possible roads to take for further research.

In Theorem~\ref{thm:sobolev} we proved that for any bounded open set
$U\subset \mathbb{R}^{d}$, the Sobolev space $W^{-1,1}(U)$ is isometric
to the Lipschitz free space over the metric quotient $\mathbb{R}^{d}/U
^{c}$. Since it is known that $\mathcal{F}(\mathbb{R}^{d})\simeq
\mathcal{F}(\mathcal{M})$ whenever $\mathcal{M}\subset \mathbb{R}^{d}$
has nonempty interior (see \cite{K15}*{Corollary 3.5}), we are tempted
to ask the following question which could be of some interest for
specialists in Sobolev spaces.

\begin{Question}
Let $d\in \mathbb{N}$, $B$ be an open unit ball in $\mathbb{R}^{d}$, and
$U\subseteq \mathbb{R}^{d}$ be a bounded open set. Are the Lipschitz
free spaces over metric quotients $\mathbb{R}^{d}/U^{c}$ and
$\mathbb{R}^{d}/B^{c}$ isomorphic?
\end{Question}

We continue with two sample questions one could consider by taking into
account what is known about copies of $\ell _{1}$ (respectively,
$\ell _{p}$) in Lipschitz free spaces (respectively, $p$-spaces). Recall
that $\ell _{1}$ is quite often isometric to a $1$-complemented subspace
of a Lipschitz free space (e.g., if the corresponding metric space has
an accumulation point) and that $\ell _{p}$ always embeds isomorphically
into a Lipschitz free $p$-space. See
Sections~\ref{Section:Comp} and \ref{Sect3} for more details, results,
and references.

\begin{Question}
\label{qu:1complemented}
For which nonseparable metric spaces $\mathcal{M}$ is $\ell _{1}(
\operatorname{dens} \mathcal{M})$ isometric to a
$1$-complemented subspace of $\mathcal{F}(\mathcal{M})$?
\end{Question}

\begin{Question}
\label{qu:1embeded}
Is $\ell _{p}$ isometric to a subspace of $\mathcal{F}_{p}(\mathcal{M})$
whenever the metric space $\mathcal{M}$ has an accumulation point?
\end{Question}

It is not very difficult to prove that if $\mathcal{N}$ is a net in a
Banach space $X$ then $\mathcal{F}(X)$ is finitely representable in
$\mathcal{F}(\mathcal{N})$ and vice versa. Now, when $p<1$ it is not
even known (see \cite{AACD2018}*{Question 6.2}) whether $
\mathcal{F}_{p}(\mathcal{N})$ isomorphically embeds into $\mathcal{F}
_{p}(X)$, so the techniques from the case $p=1$ break down.
Theorem~\ref{thm:crudely} suggests that an analogous result for $p<1$
could be true as well.

\begin{Question}
Suppose $\mathcal{N}$ is a net in an infinite-dimensional Banach space
$X$, and let $0<p<1$. Is $\mathcal{F}_{p}(\mathcal{N})$ crudely
finitely  representable in $\mathcal{F}_{p}(X)$? Is $\mathcal{F}_{p}(X)$ crudely
finitely representable in $\mathcal{F}_{p}(\mathcal{N})$?
\end{Question}

Note that, in general, $\mathcal{F}(
\mathcal{M})$ need not be  crudely finitely representable in $
\mathcal{F}(\mathcal{N})$  whenever $\mathcal{N}$ is a net in an
unbounded separable metric space $\mathcal{M}$.  Indeed, pick a separable Banach space $X$
which is not crudely finitely  representable in $\ell _{1}$ and
$x\in X\setminus \{0\}$. Consider $\mathcal{N}=\{nx\colon n\in
\mathbb{N}\}$ and $\mathcal{M}=B_{X}\cup \mathcal{N}$. Then
$\mathcal{N}$ is a net in $\mathcal{M}$ and $\mathcal{F}(\mathcal{N})
\simeq \ell _{1}$, but $\mathcal{F}(\mathcal{M})$ is not crudely
finitely representable in $\ell _{1}$ because $X$ is isomorphic to a complemented
subspace of $\mathcal{F}(X)\simeq \mathcal{F}(B_{X})$ (see
\cite{GodefroyKalton2003}*{Theorems 2.12 and 3.1} and
\cite{K15}*{Corollary 3.3}).

Finally, recall that the spaces $\mathcal{F}_{p}(\mathbb{Z}^{d})$, $d\in \mathbb{N}$, and  $\mathcal{F}_{p}([0,1])$ have
a Schauder basis (see Theorem~\ref{rem:freeZetD} and
Theorem~\ref{thm:basisFpR}). Moreover, it is known that
$\mathcal{F}([0,1]^{d})$ also has a Schauder basis for each $d\in
\mathbb{N}$ (see \cites{HP14,K15}). Thus, it is natural to ask the
following.

\begin{Question}
Does $\mathcal{F}_{p}([0,1]^{d})$ have a Schauder basis for each
$d\in \mathbb{N}$ and each $p\in (0,1]$?
\end{Question}

There are also several other areas of research one could consider and
which we left untouched in this paper. We conclude with two possible
questions on the subject of approximation properties. Let us emphasize that the
answer to both questions is positive if $p=1$ (see
\cite{GodefroyKalton2003} and \cite{Pernecka-Lancien}), but it seems
that the proofs cannot be directly generalized to the case when $p<1$.
Thus, fresh ideas are needed in order to move on in this direction. These new ideas would help to better understand
the classical case (that is, $p=1$)
as well.

\begin{Question}
Does $\mathcal{F}_{p}(X)$ have a metric approximation property whenever
$X$ is a finite dimensional Banach space and $p\in (0,1]$?
\end{Question}

\begin{Question}
Let $\mathcal{M}\subset \mathbb{R}^{d}$ and $p\in (0,1]$. Is
$\mathcal{F}_{p}(\mathcal{M})$ isomorphic to a complemented subspace of
$\mathcal{F}_{p}(\mathbb{R}^{d})$?
\end{Question}

\subsection*{Acknowledgment}
The authors would like to thank Prof. Anton\'{i}n Proch\'{a}zka for the
helpful discussions maintained during the Conference ``Non Linear
Functional Analysis'' held at CIRM (Luminy, France) from March 5 to 9,
2018. They also thank Prof. Przemys\l aw Wojtaszczyk for clarifying
remarks on the existence of unconditional bases in nonlocally convex
quasi-Banach spaces.

\end{document}